\documentclass[11pt]{article}
\usepackage[latin1]{inputenc}
\usepackage{amsmath,amsthm,amssymb}
\usepackage{amsfonts}
\usepackage{amsmath,amsthm,amssymb,amscd}
\usepackage{latexsym}
\usepackage{color}
\usepackage{graphicx}
\usepackage{mathrsfs}
\usepackage{cite}

\usepackage{color,enumitem,graphicx}
\usepackage[colorlinks=true,urlcolor=black,
citecolor=black,linkcolor=black,linktocpage,pdfpagelabels,
bookmarksnumbered,bookmarksopen]{hyperref}

\makeatletter \@addtoreset{equation}{section} \makeatother

\newtheorem{theorem}{Theorem}[section]

\newtheorem{proposition}{Proposition}[section]
\newtheorem{lemma}{Lemma}[section]
\newtheorem{remark}{Remark}[section]

\allowdisplaybreaks

\begin{document}
\title{Blowing-up solutions for the Choquard type Brezis-Nirenberg problem in dimension three}

\author{Wenjing Chen\footnote{Corresponding author.}\ \footnote{E-mail address:\, {\tt wjchen@swu.edu.cn} (W. Chen), {\tt zxwangmath@163.com} (Z. Wang).}\  \ and Zexi Wang\\
\footnotesize  School of Mathematics and Statistics, Southwest University,
Chongqing, 400715, P.R. China}

\date{ }
\maketitle

\begin{abstract}
{ In this paper, we are interested in the existence of solutions for the
following Choquard type Brezis-Nirenberg problem
\begin{align*}
 \left\{
  \begin{array}{ll}
  -\Delta u=\displaystyle\Big(\int\limits_{\Omega}\frac{u^{6-\alpha}(y)}{|x-y|^\alpha}dy\Big)u^{5-\alpha}+\lambda u,
  \ \  &\mbox{in}\ \Omega,\\
  u=0,
  \ \  &\mbox{on}\ \partial \Omega,
    \end{array}
    \right.
  \end{align*}
where $\Omega$ is a smooth bounded domain in $\mathbb{R}^3$,  $\alpha\in (0,3)$, $6-\alpha$ is the upper critical exponent in the sense of the Hardy-Littlewood-Sobolev inequality, and $\lambda$ is a real positive parameter.
By applying the reduction argument, we find and characterize a positive value $\lambda_0$ such that if $\lambda-\lambda_0>0$ is small enough, then the above problem admits a solution, which blows up and concentrates at the critical point of the Robin function as $\lambda\rightarrow \lambda_0$. Moreover, we consider the above problem under zero Neumann boundary condition.}



\end{abstract}

\section{Introduction}
In this article, we consider the following Choquard type Brezis-Nirenberg problem
\begin{align}\label{propro}
 \left\{
  \begin{array}{ll}
  -\Delta u=\displaystyle\Big(\int\limits_{\Omega}\frac{u^{6-\alpha}(y)}{|x-y|^\alpha}dy\Big)u^{5-\alpha}+\lambda u,
  \ \  &\mbox{in}\ \Omega,\\
  u=0,
  \ \  &\mbox{on}\ \partial \Omega,
    \end{array}
    \right.
  \end{align}
where  $\Omega$ is a smooth bounded domain in $\mathbb{R}^3$, $\alpha\in (0,3)$, $6-\alpha$ is the upper critical exponent in the sense of the Hardy-Littlewood-Sobolev inequality, and $\lambda$ is a real positive parameter.

In the classical paper \cite{BN}, Brezis and Nirenberg considered the following problem
\begin{align}\label{cla}
 \left\{
  \begin{array}{ll}
  -\Delta u=|u|^{2^*-2}u+\lambda u,
  \ \  &\mbox{in}\ \Omega,\\
  u=0,
  \ \  &\mbox{on}\ \partial \Omega,
    \end{array}
    \right.
  \end{align}
where $\Omega$ is a smooth bounded domain in $\mathbb{R}^N$, $N\geq 3$, $2^*=\frac{2N}{N-2}$, and $\lambda>0$ is a parameter. They proved that: if $N\geq4$, problem \eqref{cla} has a solution with minimal energy for all $\lambda\in (0,\lambda_1)$, where $\lambda_1$ is the first eigenvalue of $-\Delta$ with Dirichlet boundary condition; when $N=3$, there exists $\lambda_*\in (0,\lambda_1)$ such that \eqref{cla} has a solution with minimal energy for any $\lambda\in (\lambda_*,\lambda_1)$, and no solution with minimal energy exists for $\lambda\in (0,\lambda_*)$.
Furthermore, if $\Omega$ is a ball in $\mathbb{R}^3$, then $\lambda_*=\frac{\lambda_1}{4}$, and problem \eqref{cla} has a solution if and only if $\lambda\in (\frac{\lambda_1}{4},\lambda_1)$. The classical Poho\u{z}aev identity \cite{Po} guarantees that problem \eqref{cla} with $\lambda\leq 0$ has no solution if $\Omega$ is a star-shaped domain. In \cite{D}, Druet also showed that when $\lambda=\lambda_*$, there is no solution with minimal energy for \eqref{cla} in dimension three, which implies that $\lambda_*$ can be characterized as the critical value such that solutions of \eqref{cla} with minimal energy exist if and only if $\lambda\in (\lambda_*,\lambda_1)$. For more investigations about \eqref{cla}, we can see \cite{CFP,CSS,CW,SZ} and references therein.

In dimension three, $\lambda_*$ can be characterized by the  Robin function $g_\lambda$ defined as follows. Let $\lambda\in (0,\lambda_1)$, for any given $x\in \Omega$, consider the Green function $G_\lambda(x,y)$, solution of
\begin{align*}
 \left\{
  \begin{array}{ll}
  -\Delta_y G_\lambda(x,y)-\lambda G_\lambda(x,y)=\delta(x-y),
  \ \  &\mbox{$y\in \Omega$},\\
  G_\lambda(x,y)=0,
  \ \  &\mbox{$y\in \partial \Omega$},
    \end{array}
    \right.
  \end{align*}
where $\delta(x)$ denotes the Dirac measure at the origin. Let $H_\lambda(x,y)=\Gamma(x-y)-G_\lambda(x,y)$ with $\Gamma(z)=\frac{1}{4\pi |z|}$, be its regular part, i.e.,
$H_\lambda(x,y)$ is the unique solution of the following problem
\begin{align*}
 \left\{
  \begin{array}{ll}
  -\Delta_y H_\lambda(x,y)-\lambda H_\lambda(x,y)=-\lambda \Gamma(x-y),
  \ \  &\mbox{$y\in \Omega$},\\
  H_\lambda(x,y)=\Gamma(x-y),
  \ \  &\mbox{$y\in \partial \Omega$}.
    \end{array}
    \right.
  \end{align*}
Let us define the Robin function of $G_\lambda$ as
\begin{equation*}
  g_\lambda(x)=H_\lambda(x,x).
\end{equation*}
It follows from \cite[Lemmas A.1, A.2]{dDM} that  $g_\lambda(x)$ is a smooth function which goes to $+\infty$ as $x$ approaches to $\partial \Omega$. The minimum of $g_\lambda$ in $\Omega$ is strictly decreasing in $\lambda$, is strictly positive when $\lambda$ is close to $0$ and approaches $-\infty$ as $\lambda\rightarrow\lambda_1$.
It was conjectured in \cite{B} and proved by Druet \cite{D} that $\lambda_*$ is the largest $\lambda\in (0,\lambda_1)$ such that $\min\limits_{\Omega}g_\lambda>0$.

In the last decades, a lot of attention has been focused on the study
of
the blowing-up analysis of solutions for \eqref{cla}. On the one hand, when $N\geq4$, Rey \cite{R1} (independently and using different arguments, by Han \cite{H}) proved that if $u_\lambda$ is a solution of \eqref{cla} and satisfies $|\nabla u_\lambda|^2\rightarrow S^{\frac{N}{2}}\delta(x-x_0)$ as $\lambda\rightarrow0$, then $x_0\in \Omega$ is a critical point of the Robin function $g(x)$, where $S$ is the best Sobolev constant defined by
\begin{equation*}
   S:=\inf\limits_{u \in D^{1,2}(\mathbb{R}^N)\backslash\{0\}}\frac{\displaystyle\int_{\mathbb{R}^N}|\nabla u|^2dx}{\displaystyle\Big(\int_{\mathbb{R}^N}|u|^{2^*}dx\Big)^{\frac{1}{2^*}}}.
\end{equation*}
Here, $g(x)=H(x,x)$, $x\in \Omega$, and $H(x,y)$ is the regular part of the Green function $G(x,y)$ of
\begin{align*}
 \left\{
  \begin{array}{ll}
  -\Delta_y G(x,y)=\delta(x-y),
  \ \  &\mbox{$y\in \Omega$},\\
  G(x,y)=0,
  \ \  &\mbox{$y\in \partial \Omega$},
    \end{array}
    \right.
  \end{align*}
i.e., $H(x,y)=\Gamma(x-y)-G(x,y)$.
On the other hand, if $N\geq5$, by applying the reduction argument, Rey \cite{R1} showed that for any non-degenerate critical point of the Robin function $g(x)$, there exists a solution of \eqref{cla} that blows up and concentrates at this point as $\lambda\rightarrow 0$.
 Musso and Pistoia \cite{MP} also constructed multiple blowing-up solutions for \eqref{cla} as $\lambda\rightarrow 0$. When $N=3$, %if we regard $\lambda$ as a parameter,
 del Pino et al. \cite{dDM} proved that: if there exists $\lambda_0\in (0,\lambda_1)$ and $\xi_0\in \Omega$ such that
$\xi_0$ is a local minimizer or a non-degenerate critical point of $g_{\lambda_0}$ with value $0$, then for any $\lambda>\lambda_0$ sufficiently close to $\lambda_0$,
problem \eqref{cla} admits a blowing-up solution. Moreover, multiple blowing-up solutions for \eqref{cla} have been established by Musso and Salazar \cite{MS2}. For
more related results, we refer the readers to \cite{CLP,IV,V,MP',IV1}  and references therein.

Now, we return to the following Choquard type problem
\begin{align}\label{Hcla}
 \left\{
  \begin{array}{ll}
  -\Delta u=\displaystyle\Big(\int_{\Omega}\frac{|u(y)|^{2_\alpha^*}}{|x-y|^\alpha}dy\Big)|u|^{2_\alpha^*-2}u+\lambda u,
  \ \  &\mbox{in}\ \Omega,\\
  u=0,
  \ \  &\mbox{on}\ \partial \Omega,
    \end{array}
    \right.
  \end{align}
where $\Omega$ is a smooth bounded domain in $\mathbb{R}^N$, $N\geq 3$, $\alpha\in (0,N)$, $\lambda>0$ is a parameter, and  $2_\alpha^*=\frac{2N-\alpha}{N-2}$ is the upper critical exponent in the sense of the Hardy-Littlewood-Sobolev inequality (see Proposition \ref{HL}). Equation \eqref{Hcla} is closely related to the following nonlinear Choquard equation
\begin{equation}\label{back}
  -\Delta u+V(x)u=\displaystyle\Big(\int_{\mathbb{R}^N}\frac{|u(y)|^{p}}{|x-y|^\alpha}dy\Big)|u|^{p-2}u,
  \ \  \mbox{in}\ \mathbb{R}^N.
\end{equation}
For the case $N=3$, $\alpha=1$, $p=2$, and $V=1$, it goes back to the description of the quantum theory of a polaron at rest by Peker \cite{Pekar} and the modeling of an electron trapped in its own hole in the work of Choquard. See also \cite{MPT} for more physical background of \eqref{back}.

In recent years, much attention has been paid to study \eqref{back}, see e.g. \cite{MS1,MS2',MS3,MS4,MZ,GY} and references therein. In particular,  when $V(x)=1$, Moroz and Van Schaftingen \cite{MS2'} studied the positivity, regularity, decay behavior and radial symmetry of ground state solutions for \eqref{back}.
Meanwhile, they proved that \eqref{back} has no nontrivial solution for either $ \frac{1}{p}\leq\frac{N-2}{2N-\alpha}$ or $ \frac{1}{p}\geq\frac{N}{2N-\alpha}$ by using the Poho\u zaev identity. The number $\frac{2N-\alpha}{N}$ and $\frac{2N-\alpha}{N-2}$ (if $N\geq3$) are called the lower and upper critical exponents related to the Hardy-Littlewood-Sobolev inequality respectively.
Gao and Yang \cite{GY} studied the existence of solutions for \eqref{Hcla} and proved that: if $N\geq4$, problem \eqref{Hcla} has a solution for any $\lambda>0$; when $N=3$, there exists $\lambda^*$ such that \eqref{Hcla} has a solution for any $\lambda>\lambda^*$, where $\lambda$ is not an eigenvalue of $-\Delta$ with Dirichlet boundary condition; if $\lambda\leq0$ and $\Omega$ is a star-shaped domain, then \eqref{Hcla} admits no solution.

In \cite{YZ}, Yang and Zhao  first analyzed the blowing-up behaviour of solutions for \eqref{Hcla}, they proved that if $u_\lambda$ is a solution of \eqref{Hcla} and satisfies $|\nabla u_\lambda|^2\rightarrow S_{H,L}^{\frac{2N-\alpha}{N+2-\alpha}}\delta(x-x_0)$ as $\lambda\rightarrow0$, then $x_0\in \Omega$ is a critical point of the Robin function $g(x)$, where $N\geq4$, $S_{H,L}$ is the best Sobolev constant defined by
\begin{equation*}
   S_{H,L}:=\inf\limits_{u \in D^{1,2}(\mathbb{R}^N)\backslash\{0\}}\frac{\displaystyle\int_{\mathbb{R}^N}|\nabla u|^2dx}{\displaystyle\Big(\iint_{\mathbb{R}^{2N}}
    \frac{|u(y)|^{2_\alpha^*}|u(x)|^{2_\alpha^*}}{|x-y|^\alpha}dxdy\Big)^{\frac{1}{2_\alpha^*}}}.
\end{equation*}
Moreover, Yang et al. \cite{YYZ} provided a converse result for \cite{YZ} and obtained a solution
that blows up and concentrates at the critical point of the Robin function $g(x)$ under some suitable assumptions, if $\lambda\rightarrow 0$ and $N\geq5$.
For more related results of \eqref{Hcla}, the readers may refer to \cite{GMYZ,SYZ,YGRY,ZYY} and references therein.

Motivated by the results already mentioned above, especially \cite{dDM} and \cite{YYZ}, it is natural to ask that,
{\em does problem \eqref{Hcla} has a blowing-up solution in dimension three?} In this paper, we give an affirmative answer for this, and
our first result states as follows.
\begin{theorem}\label{th}
Assume that for a number $\lambda_0>0$, one of the following two situations holds.

$(a)$ There is an open subset $\mathfrak{D}$ of $\Omega$ such that
\begin{equation*}
  0=\inf\limits_{\mathfrak{D}}g_{\lambda_0}<\inf\limits_{\partial\mathfrak{D}}g_{\lambda_0}.
\end{equation*}

$(b)$ There is a point $\xi_0\in \Omega$ such that $g_{\lambda_0}(\xi_0)=0$, $\nabla g_{\lambda_0}(\xi_0)=0$ and $D^2g_{\lambda_0}(\xi_0)=0$ is non-singular.
\\Then for all $\lambda>\lambda_0$ sufficiently close to $\lambda_0$, there exists a solution $u_\lambda$ of problem \eqref{propro} of the form:
\begin{equation*}
  u_\lambda(x)=3^{1/4}\Big(\frac{\mu_\lambda}{\mu_\lambda^2+|x-\xi_\lambda|^2}\Big)^{1/2}+O(\mu_\lambda^{1/2}), \quad \mu_\lambda=-\gamma\frac{g_\lambda(\xi_\lambda)}{\lambda}>0,
\end{equation*}
for some $\gamma>0$. Here we have $\xi_\lambda\in \mathfrak{D}$ if case $(a)$ holds and $\xi_\lambda\rightarrow\xi_0$ as $\lambda\rightarrow \lambda_0$ if $(b)$ holds. Moreover, for some positive numbers $\beta_1,\beta_2$, we have
\begin{equation*}
  \beta_1(\lambda-\lambda_0)\leq -g_\lambda(\xi_\lambda)\leq \beta_2(\lambda-\lambda_0).
\end{equation*}
\end{theorem}

Our second result concerns the following Choquard type Lin-Ni-Takagi problem
\begin{align}\label{L}
 \left\{
  \begin{array}{ll}
  -\Delta u=\displaystyle\Big(\int_{\Omega}\frac{u^{6-\alpha}(y)}{|x-y|^\alpha}dy\Big)u^{5-\alpha}-\lambda u,
  \ \  &\mbox{in}\ \Omega,\\
  \displaystyle \frac{\partial u}{\partial \nu}=0,
  \ \  &\mbox{on}\ \partial \Omega,
    \end{array}
    \right.
  \end{align}
where $\alpha\in (0,3)$, $\lambda>0$, $\nu$ denotes the outward unit normal vector of
$\partial \Omega$, and $\Omega$ is a smooth bounded domain in $\mathbb{R}^3$.

The starting point on the study of \eqref{L} is its local version
\begin{align}\label{Lcla}
 \left\{
  \begin{array}{ll}
  -\Delta u=|u|^{p-2}u-\lambda u,
  \ \  &\mbox{in}\ \Omega,\\
 \displaystyle \frac{\partial u}{\partial \nu}=0,
  \ \  &\mbox{on}\ \partial \Omega,
    \end{array}
    \right.
  \end{align}
where $\Omega$ is a smooth bounded domain in $\mathbb{R}^N$, $N\geq 3$, $p>1$ and $\lambda>0$. The study of the zero Neumann boundary condition with Laplacian operator is a hot topic in nonlinear PDEs nowadays, and
a large literature has been devoted to study \eqref{Lcla} when $p\in [2,2^*]$. If $p\in (2,2^*)$, Lin, Ni, and Takagi \cite{LMT} proved that: as $\lambda\rightarrow0$, the only solution of \eqref{Lcla} is the constant; as $\lambda\rightarrow+\infty$, \eqref{Lcla} admits nonconstant solutions, which blow up and concentrate at one or several points. Moreover, Ni and Takagi  \cite{NT1,NT2} found that the least energy solution blows up and concentrates at a boundary point which maximizes the mean curvature of the boundary.
In the critical case, i.e., $p=2^*$, as $\lambda\rightarrow+\infty$, nonconstant solutions exist \cite{AM}, and the least energy solution blows up and concentrates at a unique point which maximizes the mean curvature of the boundary \cite{APY}. Based on the results mentioned above, Lin and Ni \cite{LM} conjectured that:
\medskip

\noindent {\bf Lin-Ni Conjecture:} If $p=2^*$, as $\lambda\rightarrow0$, problem \eqref{Lcla} admits only the constant solution.

\medskip

The above conjecture was studied by many scholars. In \cite{AY1,AY2},  Adimurthi and Yadava obtained radial solutions for \eqref{Lcla} when $\Omega$ is a ball in dimensions $N=4,5,6$, while no radial solution exists when $N=3$ or $N\geq7$. For a general convex domain, the Lin-Ni conjecture is true in dimension three \cite{WX,Z}. Wang et al. \cite{WWY} proved that this conjecture is false for all dimensions in some (partially symmetric) non-convex domains. For more classical results regarding the  Lin-Ni conjecture, we can see \cite{RW,DRW,WWY1,AY3} and references therein.

Noted that all the results mentioned above of \eqref{Lcla} are concerned with $\lambda>0$ small or large enough.
In \cite{dMRW},
del Pino et al. studied \eqref{Lcla} in dimension three and showed a new phenomenon, which is the existence of blowing-up solutions for \eqref{Lcla} when $\lambda$ closes to a number $\lambda^*\in(0,+\infty)$.
Furthermore, Salazar \cite{Sa} investigated the existence of sign-changing solutions, which blow up and concentrate at several different points.

Finally, we mention that
Giacomoni et al. \cite{GRS} first considered the following Choquard type Lin-Ni-Takagi problem
\begin{align}\label{Gia}
 \left\{
  \begin{array}{ll}
  -\Delta u=\displaystyle\Big(\int_{\Omega}\frac{|u(y)|^{2_\alpha^*}}{|x-y|^\alpha}dy\Big)|u|^{2_\alpha^*-2}u+\lambda  h(x)u,
  \ \  &\mbox{in}\ \Omega,\\
  \displaystyle \frac{\partial u}{\partial \nu}=0,
  \ \  &\mbox{on}\ \partial \Omega,
    \end{array}
    \right.
  \end{align}
where $\Omega$ is a smooth bounded domain in $\mathbb{R}^N$, $N\geq 4$, $\alpha\in (0,N)$, $\lambda>0$, $h\in C^\infty(\overline{\Omega})$ and $ \int_{\Omega}h(x)dx<0$. Under proper assumptions on $\lambda$ and $h(x)$, the authors obtained the existence of a solution for problem \eqref{Gia}.

Inspired by \cite{dMRW} and \cite{GRS},
a natural question arises,
{\em
does \eqref{Gia} has a blowing-up solution when $N=3$?} In the rest of the paper, we focuses on this issue.
Before presenting the main result, we shall make some notations.
For $\lambda>0$, we let $G^\lambda(x,y)$ be the Green function of the problem
\begin{align*}
 \left\{
  \begin{array}{ll}
  -\Delta_y G^\lambda(x,y)+\lambda G^\lambda(x,y)=\delta(x-y),
  \ \  &\mbox{$y\in \Omega$},\\
 \displaystyle  \frac{\partial G^\lambda(x,y)}{\partial \nu}=0,
  \ \  &\mbox{$y\in \partial \Omega$},
    \end{array}
    \right.
  \end{align*}
and $H^\lambda(x,y)=\Gamma(x-y)-G^\lambda(x,y)$ be its regular part, then
\begin{align*}
 \left\{
  \begin{array}{ll}
  -\Delta_y H^\lambda(x,y)+\lambda H^\lambda(x,y)=-\lambda \Gamma(x-y),
  \ \  &\mbox{$y\in \Omega$},\\
  \displaystyle  \frac{\partial H^\lambda(x,y)}{\partial \nu}=\frac{\partial \Gamma(x-y)}{\partial \nu},
  \ \  &\mbox{$y\in \partial \Omega$}.
    \end{array}
    \right.
  \end{align*}
Define the Robin function of $G^\lambda$ as
\begin{equation*}
  g^\lambda(x)=H^\lambda(x,x).
\end{equation*}
From \cite[Lemmas 2.1, 2.2]{dMRW}, we know $g^\lambda(x)$ is a smooth function which goes to $-\infty$ as $x$ approaches to $\partial \Omega$. The maximum of $g_\lambda$ in $\Omega$ is strictly increasing in $\lambda$, is strictly positive when $\lambda$ is close to $+\infty$ and approaches $-\infty$ as $\lambda\rightarrow0$.
Moreover, the number $\lambda^*$ obtained in \cite{dMRW} is the smallest $\lambda\in (0,+\infty)$ such that $\max\limits_{\Omega}g^\lambda<0$.

Our second result is as follows.
\begin{theorem}\label{th1}
Assume that for a number $\lambda^0>0$, one of the following two situations holds.

$(a)$ There is an open subset $\mathcal{U}$ of $\Omega$ such that
\begin{equation*}
  0=\sup\limits_{\mathcal{U}}g^{\lambda_0}>\sup\limits_{\partial\mathcal{U}}g^{\lambda_0}.
\end{equation*}

$(b)$ There is a point $\xi^0\in \Omega$ such that $g^{\lambda_0}(\xi^0)=0$, $\nabla g^{\lambda^0}(\xi^0)=0$ and $D^2g^{\lambda^0}(\xi^0)=0$ is non-singular.
\\Then for all $\lambda>\lambda^0$ sufficiently close to $\lambda^0$, there exists a solution $u^\lambda$ of problem \eqref{L} of the form:
\begin{equation*}
  u^\lambda(x)=3^{1/4}\Big(\frac{\mu^\lambda}{(\mu^\lambda)^2+|x-\xi^\lambda|^2}\Big)^{1/2}+O\big((\mu^\lambda)^{1/2}\big), \quad \mu^\lambda=\gamma\frac{g^\lambda(\xi^\lambda)}{\lambda}>0,
\end{equation*}
for some $\gamma>0$. Here we have $\xi^\lambda\in \mathcal{U}$ if case $(a)$ holds and $\xi^\lambda\rightarrow\xi^0$ as $\lambda\rightarrow \lambda_0$ if $(b)$ holds. Moreover, for some positive numbers $\beta_1,\beta_2$, we have
\begin{equation*}
  \beta_1(\lambda-\lambda^0)\leq g_\lambda(\xi^\lambda)\leq \beta_2(\lambda-\lambda^0).
\end{equation*}
\end{theorem}

\begin{remark}
{\rm By the definition and continuity of $g_\lambda$, it clearly follows that $\min\limits_{\Omega}g_{\lambda_*}=0$,
hence there is an open set $\mathfrak{D}$ with compact closure inside $\Omega$ such that
\begin{equation*}
  0=\inf\limits_{\mathfrak{D}}g_{\lambda_*}<\inf\limits_{\partial\mathfrak{D}}g_{\lambda_*}.
\end{equation*}
Let $\lambda_0=\lambda_*$, then $\lambda_0$ satisfies condition $(a)$ of Theorem \ref{th}. Similar arguments apply to $g^\lambda$ in Theorem \ref{th1}-$(a)$.}
\end{remark}

\begin{remark}
{\rm  Compared with the previous work, there are some features of this paper as follows:

(i) The result obtained in Theorem \ref{th} extends the earlier results of the local problem in \cite{dDM} and the high-dimensional problem ($N\geq5$) in \cite{YYZ} to the case of the nonlocal problem in dimension three.

(ii) Theorem \ref{th1} generalized the results of the local problem in \cite{dMRW} and the high-dimensional problem ($N\geq4$) in \cite{GRS} to a nonlocal one in dimension three.}
\end{remark}

\begin{remark}
{\rm Since we are working with the Choquard nonlinearity, there are some difficulties to deal with:

(i) It is difficult to calculate the norm of the nonlocal term directly. For this, we regard the nonlocal term as a operator, then by the Hardy-Littlewood-Sobolev inequality and the definition of the norm for a operator, we obtain the desired result, see e.g. Lemmas \ref{non} and \ref{err}.

(ii) Since the appearance of the nonlocal term, it is natural to make some adjustments for the projections obtained in  \cite{dDM} and \cite{dMRW}, we can see this in \eqref{pi} and \eqref{pi1}.
}
\end{remark}

\begin{remark}
{\rm In this paper, we apply the reduction argument to complete our proof, and a crucial step is to prove that the operator $T$ (defined in \eqref{map}) is a contraction map.
Different from \cite[Lemma 2.5]{YYZ}, we give a new proof for this, see the proof of Proposition \ref{fixed}.}
\end{remark}

\begin{remark}
{\rm In this paper, we focuses on the existence of single blowing-up solutions, and from \cite{MS2}, \cite{Sa}, one may ask that, does \eqref{propro} or \eqref{L} possesses multiple blowing-up solutions? This is a natural but non-obvious generalization, since there exist some interactions between bubblings, and a more precise estimate of energy expansion is needed, see e.g. \cite[Lemma 2.1]{MS2} and \cite[Lemma 2.1]{Sa}, we will study it in the forthcoming work.
}
\end{remark}

The proof of our results relies on a well known finite dimensional reduction method, introduced in \cite{BC,FW}.
The paper is organized as follows. In Section \ref{Pre}, we introduce some preliminary results. Section \ref{Energy} is devoted to the energy expansion. %and the existence of critical points for this functional under some small perturbations.
In Section \ref{Reduction}, we perform the finite dimensional reduction, and give some $C^1$-estimates in Section \ref{C^1}.
In Section \ref{Final}, we complete the proof of Theorem \ref{th}. Finally, in Section \ref{Numan}, we briefly treat problem \eqref{L} and prove Theorem \ref{th1}. Throughout the paper, $C$ denotes positive constant possibly different from line to line,
$A=o(B)$ means $A/B\rightarrow 0$ and $A=O(B)$ means that $|A/B|\leq C$.

\section{Preliminaries}\label{Pre}
In this section, we give some preliminaries. For the nonlocal problem with the convolution, an important inequality due to the Hardy-Littlewood-Sobolev inequality will be used in the following.
\begin{proposition}\cite[Theorem 4.3]{LL}\label{HL}
Let $\theta,r>1$ and $\alpha\in (0,3)$ with $\frac{1}{\theta}+\frac{\alpha}{3}+\frac{1}{r}=2$. If $f\in L^\theta(\mathbb{R}^3)$ and $g\in L^r(\mathbb{R}^3)$, then there exists a sharp constant $C(\theta,r,\alpha)$ independent of $f,g$, such that
\begin{equation}\label{HLS}
  \displaystyle\int_{\mathbb{R}^3}\int_{\mathbb{R}^3}\frac{f(x)g(y)}{|x-y|^\alpha}dxdy\leq C(\theta,r,\alpha)\|f\|_{L^\theta(\mathbb{R}^3)}\|g\|_{L^r(\mathbb{R}^3)}.
\end{equation}
If $\theta=r=\frac{6}{6-\alpha}$, then there is equality in \eqref{HLS} if and only if $f=cg$ for a constant $c$ and
\begin{equation*}
  g(x)=A(\gamma^2+|x-a|^2)^{-\frac{6-\alpha}{2}}
\end{equation*}
for some $A\in \mathbb{C}$, $0\neq\gamma\in \mathbb{R}$ and $a\in \mathbb{R}^3$.
\end{proposition}

\begin{lemma}\cite[Section 5]{Matt}
For $f,g\in L_{loc}^1(\mathbb{R}^3)$, there holds
\begin{equation}\label{CS}
  \displaystyle\int_{\mathbb{R}^3}\int_{\mathbb{R}^3}\frac{|f(x)||g(y)|}{|x-y|^\alpha}dxdy\leq \Big(\int_{\mathbb{R}^3}\int_{\mathbb{R}^3}\frac{|f(x)||f(y)|}{|x-y|^\alpha}dxdy\Big)^{\frac{1}{2}}\Big(\int_{\mathbb{R}^3}\int_{\mathbb{R}^3}\frac{|g(x)||g(y)|}{|x-y|^\alpha}dxdy\Big)^{\frac{1}{2}}.
\end{equation}
\end{lemma}

Given a  positive number $\mu$ and a point $\xi\in \mathbb{R}^3$, we denote by
\begin{equation*}
  w_{\mu,\xi}(x)=3^{1/4}\Big(\frac{\mu}{\mu^2+|x-\xi|^2}\Big)^{1/2},
\end{equation*}
which correspond to all positive solutions of
\begin{equation}\label{lim1}
  -\Delta w=w^5,\quad \text{in $\mathbb{R}^3$.}
\end{equation}
From \cite[Lemma 1.1]{YYZ}, we know $w_{\mu,\xi}$ satisfies
\begin{equation*}
  -\Delta w_{\mu,\xi}=A_{H,L}\displaystyle\Big(\int_{\mathbb{R}^3}\frac{w_{\mu,\xi}^{6-\alpha}(y)}{|x-y|^\alpha}dy\Big)w_{\mu,\xi}^{5-\alpha},\quad \text{in $\mathbb{R}^3$},
\end{equation*}
for some constant $A_{H,L}>0$. For simplicity, in the following, we will leave out the constant $A_{H,L}$, i.e.,
\begin{equation}\label{lim2}
  \displaystyle-\Delta w_{\mu,\xi}=\Big(\int_{\mathbb{R}^3}\frac{w_{\mu,\xi}^{6-\alpha}(y)}{|x-y|^\alpha}dy\Big)w_{\mu,\xi}^{5-\alpha},\quad \text{in $\mathbb{R}^3$}.
\end{equation}

In order to apply the reduction arguments, the non-degeneracy property of solution $w_{\mu,\xi}$ for \eqref{lim2} plays a crucial role. In fact, we have the following fact for the critical Choquard equation, which was established by Li et al. in \cite{LLTX} recently.

\begin{lemma}\label{ker}\cite[Theorem 1.5]{LLTX}
Let $\alpha\in (0,3)$, then the kernel of the linear operator for \eqref{lim2} at $w_{\mu,\xi}$
\begin{equation*}
 \ell(h)=\displaystyle -\Delta h-(6-\alpha)\Big(\int_{\mathbb{R}^3}\frac{w_{\mu,\xi}^{5-\alpha}(y)h(y)}{|x-y|^\alpha}dy\Big)w_{\mu,\xi}^{5-\alpha}-
 (5-\alpha)\Big(\int_{\mathbb{R}^3}\frac{w_{\mu,\xi}^{6-\alpha}(y)}{|x-y|^\alpha}dy\Big)w_{\mu,\xi}^{4-\alpha}h,\quad h\in D^{1,2}(\mathbb{R}^3),
\end{equation*}
is given by
\begin{equation*}
  {\bf span}\Big\{\frac{\partial w_{\mu,\xi}}{\partial \xi_1},\frac{\partial w_{\mu,\xi}}{\partial \xi_2},\frac{\partial w_{\mu,\xi}}{\partial \xi_3},\frac{\partial w_{\mu,\xi}}{\partial \mu}\Big\}.
\end{equation*}
\end{lemma}

The solutions we look for in Theorem \ref{th} have the form $u_\lambda(x)\sim w_{\mu,\xi}$, where $\mu$ is a small  positive number and $\xi\in \Omega$. It is naturally to correct this initial approximation by a term that provides Dirichlet boundary condition.
We define $\pi_{\mu,\xi}$ to be the unique solution of the problem
\begin{align}\label{pi}
 \left\{
  \begin{array}{ll}
  -\Delta \pi_{\mu,\xi}=\lambda\pi_{\mu,\xi}+\lambda w_{\mu,\xi}-\displaystyle\Big(\int_{\mathbb{R}^3\backslash \Omega}\frac{w_{\mu,\xi}^{6-\alpha}(y)}{|x-y|^\alpha}dy\Big)w_{\mu,\xi}^{5-\alpha},
  \ \  &\mbox{in}\ \Omega,\\
  \pi_{\mu,\xi}=-w_{\mu,\xi},
  \ \  &\mbox{on}\ \partial \Omega.
    \end{array}
    \right.
  \end{align}
Fix a small  positive number $\mu$ and a point $\xi\in \Omega$, we consider a first approximation of the solution of the form:
\begin{equation*}
  U_{\mu,\xi}(x)=w_{\mu,\xi}(x)+\pi_{\mu,\xi}(x).
\end{equation*}
Then $U=U_{\mu,\xi}$ satisfies the equation
\begin{align}\label{you}
 \left\{
  \begin{array}{ll}
  -\Delta U=\displaystyle\Big(\int_{\Omega}\frac{w_{\mu,\xi}^{6-\alpha}(y)}{|x-y|^\alpha}dy\Big)w_{\mu,\xi}^{5-\alpha}+\lambda U,
  \ \  &\mbox{in}\ \Omega,\\
  U=0,
  \ \  &\mbox{on}\ \partial \Omega.
    \end{array}
    \right.
  \end{align}

\section{Energy expansion}\label{Energy}
Solutions to \eqref{propro} correspond to critical points of the following energy functional
\begin{equation*}
  \mathcal{J}_\lambda(u)=\frac{1}{2}\int_{\Omega}|\nabla u|^2 dx-\frac{\lambda}{2}\int_{\Omega} u^2 dx-\frac{1}{2(6-\alpha)}\int_{\Omega}\int_{ \Omega}\frac{u^{6-\alpha}(y)u^{6-\alpha}(x)}{|x-y|^\alpha}dxdy.
\end{equation*}
Since we are looking for solutions close to $U_{\mu,\xi}$, formally, we expect $\mathcal{J}_\lambda(U_{\mu,\xi})$ to be almost critical in the parameters $\mu,\xi$. For this reason, it is important to obtain an asymptotic formula of the function $(\mu,\xi)\mapsto \mathcal{J}_\lambda(U_{\mu,\xi})$ as $\mu\rightarrow0$.

\begin{proposition}\label{enex}
For any $\sigma>0$, as $\mu\rightarrow0$, the following expansion holds:
\begin{equation*}
  \mathcal{J}_\lambda(U_{\mu,\xi})=a_0+a_1\mu g_\lambda(\xi)+a_2\lambda\mu^2 -a_3\mu^2 g_\lambda^2(\xi)+\mu^{\frac{5}{2}-\sigma}\theta(\mu,\xi),
\end{equation*}
for $i=0,1$, $j=0,1,2$, $i+j\leq 2$, and the function $\mu^j\frac{\partial ^{i+j}}{\partial \xi^i \partial \mu^j}\theta(\mu,\xi)$ is bounded uniformly on all small $\mu$ and $\xi$ in compact subsets of $\Omega$. The $a_j$'s are explicit constants, given by \eqref{changshu}.
\end{proposition}

To prove this Proposition, we need some preliminary results. To begin with, we recall the relationship between $\pi_{\mu,\xi}$ and $H_\lambda(x,\xi)$.
Let us consider the unique radial solution $\mathcal{D}_0(z)$ of the problem
\begin{align*}
 \left\{
  \begin{array}{ll}
  \displaystyle-\Delta \mathcal{D}_0=\lambda3^{1/4}\Big(\frac{1}{\sqrt{1+|z|^2}}-\frac{1}{|z|}\Big),\quad &\text{in $\mathbb{R}^3$}\\
  \mathcal{D}_0(z)\rightarrow 0,
  \ \  &\mbox{as $|z|\rightarrow +\infty$}.
    \end{array}
    \right.
  \end{align*}
Then $\mathcal{D}_0(z)$ is a $C^{0,1}$ function with $\mathcal{D}_0(z)\sim |z|^{-1}\log |z|$ as $|z|\rightarrow+\infty$.

\begin{lemma}\label{piex}
For any $\sigma>0$, as $\mu\rightarrow0$, the following expansion holds:
\begin{equation*}
  \mu^{-1/2}\pi_{\mu,\xi}(x)=-4\pi3^{1/4}H_\lambda(x,\xi)+\mu \mathcal{D}_0(\mu^{-1}({x-\xi}))+\mu^{2-\sigma}\theta(\mu,x,\xi),
\end{equation*}
for $i=0,1$, $j=0,1,2$, $i+j\leq 2$, and the function $\mu^j\frac{\partial ^{i+j}}{\partial \xi^i \partial \mu^j}\theta(\mu,x,\xi)$ is bounded uniformly on $x\in \Omega$, all small $\mu$ and $\xi$ in compact subsets of $\Omega$.
\end{lemma}
\begin{proof}
For any $\varphi\in H_0^1(\Omega)$, using \eqref{HLS}, the H\"{o}lder and Sobolev inequalities, we have
\begin{equation*}
  \bigg|\int_{\Omega}\int_{\mathbb{R}^3\backslash \Omega}\frac{w_{\mu,\xi}^{6-\alpha}(y)w_{\mu,\xi}^{5-\alpha}(x)\varphi(x)}{|x-y|^\alpha}dxdy\bigg|\leq C\Big(\int_{\mathbb{R}^3\backslash\Omega}w_{\mu,\xi}^{6}dx\Big)^{\frac{6-\alpha}{6}}\Big(\int_{\Omega}w_{\mu,\xi}^{\frac{6(5-\alpha)}{6-\alpha}}\varphi^{\frac{6}{6-\alpha}}dx\Big)^{\frac{6-\alpha}{6}}
  \leq C\mu^{\frac{6-\alpha}{2}}\|\varphi\|_{H_0^1(\Omega)}.
\end{equation*}
Hence, we obtain
\begin{equation*}
  \bigg\|\displaystyle\Big(\int_{\mathbb{R}^3\backslash \Omega}\frac{w_{\mu,\xi}^{6-\alpha}(y)}{|x-y|^\alpha}dy\Big)w_{\mu,\xi}^{5-\alpha}\bigg\|_{H_0^1(\Omega)}\leq C\mu^{\frac{6-\alpha}{2}},
\end{equation*}
and
\begin{align*}
 \left\{
  \begin{array}{ll}
  -\Delta \pi_{\mu,\xi}=\lambda\pi_{\mu,\xi}+\lambda w_{\mu,\xi}+O(\mu^{\frac{6-\alpha}{2}}),
  \ \  &\mbox{in}\ \Omega,\\
  \pi_{\mu,\xi}=-w_{\mu,\xi},
  \ \  &\mbox{on}\ \partial \Omega.
    \end{array}
    \right.
  \end{align*}
Set $\mathcal{D}_1(x)=\mu \mathcal{D}_0(\mu^{-1}(x-\xi))$, then
\begin{align*}
 \left\{
  \begin{array}{ll}
  -\Delta \mathcal{D}_1=\lambda\big(\mu^{-1/2}w_{\mu,\xi}(x)-4\pi 3^{1/4}\Gamma(x-\xi)\big),%3^{1/4}\Big(\frac{1}{\sqrt{1+|z|^2}}-\frac{1}{|z|}\Big),
  \quad &\text{in $\Omega$},\\
  \mathcal{D}_1\sim \mu^2\log \mu\quad as \,\,\mu\rightarrow 0,
  \ \  &\mbox{on $\partial \Omega$}.
    \end{array}
    \right.
  \end{align*}
Let us write
\begin{equation*}
  S_1(x)=\mu^{-1/2}\pi_{\mu,\xi}(x)+4\pi3^{1/4}H_\lambda(x,\xi)-\mathcal{D}_1(x).
\end{equation*}
With the notations of Lemma \ref{piex}, this means
\begin{equation*}
  S_1(x)=\mu^{2-\sigma}\theta(\mu,x,\xi).
\end{equation*}
Observe that for $y\in \partial \Omega$, as $\mu\rightarrow0$, we have
\begin{equation*}
  \mu^{-1/2}\pi_{\mu,\xi}(x)+4\pi3^{1/4}H_\lambda(x,\xi)=3^{1/4}\Big(\frac{1}{\sqrt{|\mu|^2+|x-\xi|^2}}-\frac{1}{|x-\xi|}\Big)\sim \mu^2|x-\xi|^{-3}.
\end{equation*}
Using the above equations, we find that $S_1$ satisfies
\begin{align*}
 \left\{
  \begin{array}{ll}
   \Delta S_1+\lambda S_1=-\lambda \mathcal{D}_1+O(\mu^{\frac{5-\alpha}{2}})=:\mathcal{D}_2,
  \ \  &\mbox{in}\ \Omega,\\
  S_1=O(\mu^2\log \mu)\quad as \,\,\mu\rightarrow 0,
  \ \  &\mbox{on $\partial \Omega$}.
    \end{array}
    \right.
  \end{align*}
For any $p>3$, we have
\begin{equation*}
  \int_{\Omega}|\mathcal{D}_1(x)|^pdx\leq \mu^{p+3}\int_{\mathbb{R}^3}|\mathcal{D}_0(z)|^pdz,
\end{equation*}
so $\|\mathcal{D}_2\|_{L^p(\Omega)}\leq C_p\mu^{(p+3)/p}+C\mu^{\frac{5-\alpha}{2}}$. Since $\alpha\in (0,3)$, applying elliptic estimates (see \cite{GT}), we know that, for any $\sigma>0$,
$\|S_1\|_{L^{\infty}(\Omega)}=O(\mu^{2-\sigma})$ uniformly on $\xi$ in compact subsets of $\Omega$. This yields the assertion of the lemma for $i,j=0$.

We now consider the quantity $S_2=\partial_{\xi}S_1$. Observe that $S_2$ satisfies
\begin{align*}
 \left\{
  \begin{array}{ll}
   \Delta S_2+\lambda S_2=-\lambda \partial_{\xi}\mathcal{D}_1,
  \ \  &\mbox{in}\ \Omega,\\
  S_2=O(\mu^2\log \mu)\quad as \,\,\mu\rightarrow 0,
  \ \  &\mbox{on $\partial \Omega$}.
    \end{array}
    \right.
  \end{align*}
Since $\partial_{\xi}\mathcal{D}_1(x)=-\nabla \mathcal{D}_0(\mu^{-1}(x-\xi))$, for any $p>3$, we have
\begin{equation*}
  \int_{\Omega}|\partial_{\xi}\mathcal{D}_1(x)|^pdx\leq \mu^{p+3}\int_{\mathbb{R}^3}|\nabla\mathcal{D}_0(z)|^pdz.
\end{equation*}
We conclude that $\|S_2\|_{L^{\infty}(\Omega)}=O(\mu^{2-\sigma})$ for any $\sigma>0$. This gives the proof of the lemma for $i=1$, $j=0$. Let us set
$S_3=\mu\partial_{\mu}S_1$, then
\begin{align*}
 \left\{
  \begin{array}{ll}
   \Delta S_3+\lambda S_3=-\lambda \mu\partial_{\mu}\mathcal{D}_1+O(\mu^{\frac{5-\alpha}{2}})=:\mathcal{D}_3,
  \ \  &\mbox{in}\ \Omega,\\
  S_3=O(\mu^2\log \mu)\quad as \,\,\mu\rightarrow 0,
  \ \  &\mbox{on $\partial \Omega$}.
    \end{array}
    \right.
  \end{align*}
Observed that
\begin{equation*}
  \mu\partial_{\mu}\mathcal{D}_1=\mu(\mathcal{D}_0+\tilde{\mathcal{D}}_0)(\mu^{-1}(x-\xi)),
\end{equation*}
where $\tilde{\mathcal{D}}_0(z)=z\cdot\nabla \mathcal{D}_0(z)$. Thus, similar to the estimate for $S_1$, we obtain
$\|S_3\|_{L^{\infty}(\Omega)}=O(\mu^{2-\sigma})$ for any $\sigma>0$. This yields the assertion of the lemma for $i=0$, $j=1$.
The proof of the remaining estimates comes after applying again $\mu \partial_\mu$ to the equations obtained for $S_2$ and $S_3$, and the desired result comes after exactly the similar arguments.
This concludes the proof.
\end{proof}
\noindent {\bf Proof of Proposition \ref{enex}.}
Let us decompose:
\begin{equation*}
  \mathcal{J}_\lambda(U_{\mu,\xi})=I+II+III+IV+V+VI,
\end{equation*}
\begin{align*}
  I=&\frac{1}{2}\int_{\Omega}|\nabla w_{\mu,\xi}|^2 dx-\frac{1}{2(6-\alpha)}\int_{\Omega}\int_{ \Omega}\frac{w_{\mu,\xi}^{6-\alpha}(y)w_{\mu,\xi}^{6-\alpha}(x)}{|x-y|^\alpha}dxdy,
  \\II=&\int_{\Omega} \nabla w_{\mu,\xi}\cdot \nabla \pi_{\mu,\xi}-\int_{\Omega}\int_{ \Omega}\frac{w_{\mu,\xi}^{6-\alpha}(y)w_{\mu,\xi}^{5-\alpha}(x)\pi_{\mu,\xi}(x)}{|x-y|^\alpha}dxdy,\\
  III=&\frac{1}{2}\int_{\Omega}|\nabla \pi_{\mu,\xi}|^2 dx-\frac{\lambda}{2}\int_{\Omega} (w_{\mu,\xi}+\pi_{\mu,\xi})\pi_{\mu,\xi} dx,\\
  IV=&-\frac{\lambda}{2}\int_{\Omega} (w_{\mu,\xi}+\pi_{\mu,\xi})w_{\mu,\xi} dx,\\
  V=&-\frac{5-\alpha}{2}\int_{\Omega}\int_{ \Omega}\frac{w_{\mu,\xi}^{6-\alpha}(y)w_{\mu,\xi}^{4-\alpha}(x)\pi^2_{\mu,\xi}(x)}{|x-y|^\alpha}dxdy
  -\frac{6-\alpha}{2}\int_{\Omega}\int_{ \Omega}\frac{w_{\mu,\xi}^{5-\alpha}(y)\pi_{\mu,\xi}(y)w_{\mu,\xi}^{5-\alpha}(x)\pi_{\mu,\xi}(x)}{|x-y|^\alpha}dxdy,\\
  VI=&-\frac{1}{2(6-\alpha)}\int_{\Omega}\int_{ \Omega}\frac{Long}{|x-y|^\alpha}dxdy,
\end{align*}
where
\begin{align*}
  Long=&(w_{\mu,\xi}+\pi_{\mu,\xi})^{6-\alpha}(y)(w_{\mu,\xi}+\pi_{\mu,\xi})^{6-\alpha}(x)-w_{\mu,\xi}^{6-\alpha}(y)w_{\mu,\xi}^{6-\alpha}(x)-2(6-\alpha)w_{\mu,\xi}^{6-\alpha}(y)w_{\mu,\xi}^{5-\alpha}(x)\pi_{\mu,\xi}(x)
 \\& -(6-\alpha)(5-\alpha)w_{\mu,\xi}^{6-\alpha}(y)w_{\mu,\xi}^{4-\alpha}(x)\pi^2_{\mu,\xi}(x)
  -(6-\alpha)^2w_{\mu,\xi}^{5-\alpha}(y)\pi_{\mu,\xi}(y)w_{\mu,\xi}^{5-\alpha}(x)\pi_{\mu,\xi}(x).
\end{align*}
Multiplying  \eqref{lim2}
by $w_{\mu,\xi}$ and integrating by parts in $\Omega$, by \eqref{HLS} and \eqref{lim1}, we obtain
\begin{align*}
  I=&\frac{1}{2}\int_{\partial\Omega}\frac{\partial w_{\mu,\xi}}{\partial \nu}w_{\mu,\xi}dx+\frac{5-\alpha}{2(6-\alpha)}\int_{\Omega}\int_{ \mathbb{R}^3}\frac{w_{\mu,\xi}^{6-\alpha}(y)w_{\mu,\xi}^{6-\alpha}(x)}{|x-y|^\alpha}dxdy+\frac{1}{2(6-\alpha)}\int_{\Omega}\int_{ \mathbb{R}^3\backslash\Omega}\frac{w_{\mu,\xi}^{6-\alpha}(y)w_{\mu,\xi}^{6-\alpha}(x)}{|x-y|^\alpha}dxdy\\
  =&\frac{1}{2}\int_{\partial\Omega}\frac{\partial w_{\mu,\xi}}{\partial \nu}w_{\mu,\xi}dx+\frac{5-\alpha}{2(6-\alpha)}\int_{\mathbb{R}^3}\int_{ \mathbb{R}^3}\frac{w_{\mu,\xi}^{6-\alpha}(y)w_{\mu,\xi}^{6-\alpha}(x)}{|x-y|^\alpha}dxdy+O(\mu^3)\\
  =&\frac{1}{2}\int_{\partial\Omega}\frac{\partial w_{\mu,\xi}}{\partial \nu}w_{\mu,\xi}dx+\frac{5-\alpha}{2(6-\alpha)}\int_{\mathbb{R}^3}w^6_{\mu,\xi}dx+O(\mu^3),
\end{align*}
where $\nu$ denotes the outward unit normal vector of
$\partial \Omega$. Testing \eqref{lim2}
against $\pi_{\mu,\xi}$, by Lemma \ref{piex},
 we find
\begin{align*}
  II=-\int_{\partial\Omega}\frac{\partial w_{\mu,\xi}}{\partial \nu}w_{\mu,\xi}dx+\int_{\Omega}\int_{ \mathbb{R}^3\backslash \Omega}\frac{w_{\mu,\xi}^{6-\alpha}(y)w_{\mu,\xi}^{5-\alpha}(x)\pi_{\mu,\xi}(x)}{|x-y|^\alpha}dxdy.
  =&-\int_{\partial\Omega}\frac{\partial w_{\mu,\xi}}{\partial \nu}w_{\mu,\xi}dx+O(\mu^{\frac{5}{2}}).
\end{align*}
Testing \eqref{pi} against $\pi_{\mu,\xi}$, we get
\begin{align*}
  III=-\frac{1}{2}\int_{\partial\Omega}\frac{\partial \pi_{\mu,\xi}}{\partial \nu}w_{\mu,\xi}dx-\frac{1}{2}\int_{\Omega}\int_{ \mathbb{R}^3\backslash \Omega}\frac{w_{\mu,\xi}^{6-\alpha}(y)w_{\mu,\xi}^{5-\alpha}(x)\pi_{\mu,\xi}(x)}{|x-y|^\alpha}dxdy=-\frac{1}{2}\int_{\partial\Omega}\frac{\partial \pi_{\mu,\xi}}{\partial \nu}w_{\mu,\xi}dx+O(\mu^{\frac{5}{2}}).
\end{align*}
Multiplying \eqref{you} by $\pi_{\mu,\xi}$,  we get
\begin{align*}
  IV=&\frac{1}{2}\int_{\partial\Omega}\frac{\partial U_{\mu,\xi}}{\partial \nu}w_{\mu,\xi}dx-\frac{1}{2}\int_{\Omega}\int_{ \Omega}\frac{w_{\mu,\xi}^{6-\alpha}(y)w_{\mu,\xi}^{5-\alpha}(x)\pi_{\mu,\xi}(x)}{|x-y|^\alpha}dxdy-\frac{1}{2}\int_{\Omega}\int_{ \mathbb{R}^3\backslash \Omega}\frac{w_{\mu,\xi}^{6-\alpha}(y)w_{\mu,\xi}^{5-\alpha}(x)U_{\mu,\xi}(x)}{|x-y|^\alpha}dxdy\\
  =&\frac{1}{2}\int_{\partial\Omega}\frac{\partial U_{\mu,\xi}}{\partial \nu}w_{\mu,\xi}dx-\frac{1}{2}\int_{\Omega}\int_{ \mathbb{R}^3}\frac{w_{\mu,\xi}^{6-\alpha}(y)w_{\mu,\xi}^{5-\alpha}(x)\pi_{\mu,\xi}(x)}{|x-y|^\alpha}dxdy+O(\mu^{\frac{5}{2}})\\
  =&\frac{1}{2}\int_{\partial\Omega}\frac{\partial U_{\mu,\xi}}{\partial \nu}w_{\mu,\xi}dx-\frac{1}{2}\int_{\Omega}w_{\mu,\xi}^{5}\pi_{\mu,\xi}dx+O(\mu^{\frac{5}{2}}).
\end{align*}
And
\begin{align*}
  V=&-\frac{5-\alpha}{2}\int_{\Omega}\int_{ \Omega}\frac{w_{\mu,\xi}^{6-\alpha}(y)w_{\mu,\xi}^{4-\alpha}(x)\pi^2_{\mu,\xi}(x)}{|x-y|^\alpha}dxdy
  -\frac{6-\alpha}{2}\int_{\Omega}\int_{ \Omega}\frac{w_{\mu,\xi}^{5-\alpha}(y)\pi_{\mu,\xi}(y)w_{\mu,\xi}^{5-\alpha}(x)\pi_{\mu,\xi}(x)}{|x-y|^\alpha}dxdy\\
  =&-\frac{5-\alpha}{2}\int_{\Omega}w_{\mu,\xi}^{4}\pi^2_{\mu,\xi}dx-\frac{6-\alpha}{2}\int_{\Omega}\int_{ \Omega}\frac{w_{\mu,\xi}^{5-\alpha}(y)\pi_{\mu,\xi}(y)w_{\mu,\xi}^{5-\alpha}(x)\pi_{\mu,\xi}(x)}{|x-y|^\alpha}dxdy+O(\mu^{\frac{7}{2}}).
\end{align*}
As for $VI$, by \eqref{HLS}-\eqref{lim2}, we have
\begin{align*}
  |VI|\leq & C\bigg|\int_{\Omega}\int_{ \Omega}\frac{w_{\mu,\xi}^{6-\alpha}(y)w_{\mu,\xi}^{3-\alpha}(x)\pi^3_{\mu,\xi}(x)}{|x-y|^\alpha}dxdy\bigg|+C\bigg|\int_{\Omega}\int_{ \Omega}\frac{w_{\mu,\xi}^{5-\alpha}(y)\pi_{\mu,\xi}(y)w_{\mu,\xi}^{4-\alpha}(x)\pi^2_{\mu,\xi}(x)}{|x-y|^\alpha}dxdy\bigg|\\
  \leq & C\Big|\int_{\Omega}w^3_{\mu,\xi}\pi^3_{\mu,\xi}dx\Big|%\Big)^{\frac{6-\alpha}{6}}
  +C\Big(\int_{\Omega}w^{\frac{6(5-\alpha)}{6-\alpha}}_{\mu,\xi}\pi^{\frac{6}{6-\alpha}}_{\mu,\xi}dx\Big)^{\frac{6-\alpha}{6}}
  \Big(\int_{\Omega}w^{\frac{6(4-\alpha)}{6-\alpha}}_{\mu,\xi}\pi^{\frac{12}{6-\alpha}}_{\mu,\xi}dx\Big)^{\frac{6-\alpha}{6}}\\
  =&C\mu^3\Big|\int_{\Omega_\mu}w^{3}_{1,0}(z)\big[\mu^{-1/2}\pi_{\mu,\xi}(\xi+\mu z)\big]^{3}dz\Big|%\Big)^{\frac{6-\alpha}{6}}
  \\&+C\mu^3\Big(\int_{\Omega_\mu}w^{\frac{6(5-\alpha)}{6-\alpha}}_{1,0}(z)\big[\mu^{-1/2}\pi_{\mu,\xi}(\xi+\mu z)\big]^{\frac{6}{6-\alpha}}dz\Big)^{\frac{6-\alpha}{6}}\Big(\int_{\Omega_\mu}w^{\frac{6(4-\alpha)}{6-\alpha}}_{1,0}(z)\big[\mu^{-1/2}\pi_{\mu,\xi}(\xi+\mu z)\big]^{\frac{12}{6-\alpha}}dz\Big)^{\frac{6-\alpha}{6}}\\
  \leq &C \mu^{\frac{5}{2}},
\end{align*}
where $\Omega_\mu={\mu}^{-1}(\Omega-\xi)$.

From \cite[Lemma 2.1]{dDM}, we know
\begin{equation*}
  \int_{\Omega}w_{\mu,\xi}^{5}\pi_{\mu,\xi}dx=-4\pi3^{1/4}\mu g_\lambda(\xi)\int_{\mathbb{R}^3}w_{1,0}^5(x)dx-3^{1/4}\lambda\mu^2\int_{\mathbb{R}^3}
  \Big[w_{1,0}(x)\Big(\frac{1}{|x|}-\frac{1}{\sqrt{1+|x|^2}}\Big)+\frac{1}{2}w_{1,0}^5(x)|x|\Big]dx+R_1,
\end{equation*}
\begin{equation*}
  \int_{\Omega}w_{\mu,\xi}^{4}\pi^2_{\mu,\xi}dx=16\pi^23^{1/2}\mu^2 g^2_\lambda(\xi)\int_{\mathbb{R}^3}w_{1,0}^4(x)dx+R_2,
\end{equation*}
with
\begin{equation*}
  \mu^j\frac{\partial ^{i+j}}{\partial \xi^i \partial \mu^j}R_l=O(\mu^{3-\sigma}),
\end{equation*}
for $l=1,2$, $i=0,1$, $j=0,1,2$, $i+j\leq 2$, uniformly on all small $\mu$ and $\xi$ in compact subsets of $\Omega$.
Moreover, by Lemma \ref{piex}, we have the following expansion
\begin{align*}
  \mu^{-1/2}\pi_{\mu,\xi}(\xi+\mu z)=&-4\pi3^{1/4}g_\lambda(\xi)-\frac{ 3^{1/4}\lambda \mu}{2}|z|-4\pi3^{1/4}\theta_1(\xi,\xi+\mu z)
+\mu \mathcal{D}_0(z)+\mu^{2-\sigma}\theta(\mu,\xi+\mu z,\xi)\\
=:&-4\pi3^{1/4}g_\lambda(\xi)+\delta_\mu(z),
\end{align*}
where $\theta_1$ is a function of class $C^2$ with $\theta_1(\xi,\xi)=0$.
From these facts, we obtain
\begin{align*}
 & \int_{\Omega}\int_{ \Omega}\frac{w_{\mu,\xi}^{5-\alpha}(y)\pi_{\mu,\xi}(y)w_{\mu,\xi}^{5-\alpha}(x)\pi_{\mu,\xi}(x)}{|x-y|^\alpha}dxdy\\
  =&
  \mu^2\int_{\Omega_\mu}\int_{ \Omega_\mu}\frac{w_{1,0}^{5-\alpha}(y')\mu^{-1/2}\pi_{\mu,\xi}(\xi+\mu y')w_{1,0}^{5-\alpha}(x')\mu^{-1/2}\pi_{\mu,\xi}(\xi+\mu x')}{|x'-y'|^\alpha}dx'dy'
  \\
  =&16\pi^23^{1/2}\mu^2 g^2_\lambda(\xi)\int_{\mathbb{R}^3}\int_{\mathbb{R}^3}\frac{w_{1,0}^{5-\alpha}(y')w_{1,0}^{5-\alpha}(x')}{|x'-y'|^\alpha}dx'dy'+R_3,
\end{align*}
where
\begin{align*}
  R_3=&-2\mu^2\int_{\Omega_\mu}\int_{\Omega_\mu}\frac{w_{1,0}^{5-\alpha}(y')\delta_\mu(y')w_{1,0}^{5-\alpha}(x')4\pi3^{1/4} g_\lambda(\xi)}{|x'-y'|^\alpha}dx'dy'\\
  &+\mu^2\int_{\Omega_\mu}\int_{\Omega_\mu}\frac{w_{1,0}^{5-\alpha}(y')\delta_\mu(y')w_{1,0}^{5-\alpha}(x')\delta_\mu(x')}{|x'-y'|^\alpha}dx'dy'
\\
&-\Big(16\pi^23^{1/2}\mu^2 g^2_\lambda(\xi)\int_{\mathbb{R}^3}\int_{\mathbb{R}^3}\frac{w_{1,0}^{5-\alpha}(y')w_{1,0}^{5-\alpha}(x')}{|x'-y'|^\alpha}dx'dy'\\&-16\pi^23^{1/2}\mu^2 g^2_\lambda(\xi)\int_{ \Omega_\mu}\int_{ \Omega_\mu}\frac{w_{1,0}^{5-\alpha}(y')w_{1,0}^{5-\alpha}(x')}{|x'-y'|^\alpha}dx'dy'\Big)\\
=:&-R_{31}+R_{32}-R_{33}.
\end{align*}
By \eqref{lim1}, \eqref{lim2} and the elementary inequality, we know
\begin{align*}
  |R_{32}|\leq& \frac{\mu^2}{2}\Big(\int_{\Omega_\mu}\int_{\Omega_\mu}\frac{w_{1,0}^{6-\alpha}(y')w_{1,0}^{4-\alpha}(x')\delta^2_\mu(x')}{|x'-y'|^\alpha}dx'dy'
  +\int_{\Omega_\mu}\int_{\Omega_\mu}\frac{w_{1,0}^{4-\alpha}(y')\delta^2_\mu(y')w_{1,0}^{6-\alpha}(x')}{|x'-y'|^\alpha}dx'dy'\Big)\\
  = \mu^2&\int_{\Omega_\mu}\int_{\Omega_\mu}\frac{w_{1,0}^{6-\alpha}(y')w_{1,0}^{4-\alpha}(x')\delta^2_\mu(x')}{|x'-y'|^\alpha}dx'dy'
  \leq\mu^2\int_{\Omega_\mu}w_{1,0}^{4}\delta^2_\mu dx\leq C|R_2|.
\end{align*}
This with \eqref{CS} yields that
\begin{equation*}
  |R_{31}|\leq C\mu |R_2|^{1/2}\Big( \int_{\mathbb{R}^3}\int_{\mathbb{R}^3}\frac{w_{1,0}^{5-\alpha}(y')w_{1,0}^{5-\alpha}(x')}{|x'-y'|^\alpha}dx'dy'\Big)^{\frac{1}{2}}\leq C \mu|R_2|^{1/2}.
\end{equation*}
Besides, using \eqref{HLS}, we have
\begin{equation*}
  |R_{33}|\leq C\mu^2\Big(\int_{\mathbb{R}^3\backslash \Omega_\mu}w_{1,0}^{\frac{6(5-\alpha)}{6-\alpha}}dx\Big)^{\frac{6-\alpha}{3}},
\end{equation*}
a similar argument of \cite[Lemma 2.1]{dDM} shows that
\begin{equation*}
  \mu^j\frac{\partial ^{i+j}}{\partial \xi^i \partial \mu^j}R_{33}=O(\mu^{\frac{5}{2}-\sigma}),
\end{equation*}
for $i=0,1$, $j=0,1,2$, $i+j\leq 2$, uniformly on all small $\mu$ and $\xi$ in compact subsets of $\Omega$.
Thus, we have
\begin{equation*}
  \mu^j\frac{\partial ^{i+j}}{\partial \xi^i \partial \mu^j}R_3=O(\mu^{\frac{5}{2}-\sigma}),
\end{equation*}
for $i=0,1$, $j=0,1,2$, $i+j\leq 2$, uniformly on all small $\mu$ and $\xi$ in compact subsets of $\Omega$.

Therefore, By \eqref{lim1}, \eqref{lim2} and the definition of $S_{H,L}$, we get
\begin{equation*}
  \mathcal{J}_\lambda(U_{\mu,\xi})=a_0+a_1\mu g_\lambda(\xi)+a_2\lambda\mu^2 -a_3\mu^2 g_\lambda^2(\xi)+\mu^{\frac{5}{2}-\sigma}\theta(\mu,\xi),
\end{equation*}
where for $i=0,1$, $j=0,1,2$, $i+j\leq 2$, the function $\mu^j\frac{\partial ^{i+j}}{\partial \xi^i \partial \mu^j}\theta(\mu,\xi)$ is bounded uniformly on all small $\mu$ and $\xi$ in compact subsets of $\Omega$,
and
\begin{align}\label{changshu}
 \left\{
  \begin{array}{ll}
  a_0=& \displaystyle\frac{5-\alpha}{2(6-\alpha)}S_{H,L}^{\frac{6-\alpha}{5-\alpha}},\\
 a_1=&\displaystyle2\pi3^{1/4}\int_{\mathbb{R}^3}w_{1,0}^5(x)dx,\\
 a_2=&\displaystyle\frac{3^{1/4}}{2}\int_{\mathbb{R}^3}
  \Big[w_{1,0}(x)\Big(\frac{1}{|x|}-\frac{1}{\sqrt{1+|x|^2}}\Big)+\frac{1}{2}w_{1,0}^5(x)|x|\Big]dx,\\
  a_3=&\displaystyle8(5-\alpha)\pi^23^{1/2}\int_{\mathbb{R}^3}w_{1,0}^4(x)dx+8(6-\alpha)\pi^23^{1/2}\int_{\mathbb{R}^3}\int_{\mathbb{R}^3}\frac{w_{1,0}^{5-\alpha}(y)w_{1,0}^{5-\alpha}(x)}{|x-y|^\alpha}dxdy.
    \end{array}
    \right.
 \end{align}
This ends the proof of Lemma \ref{enex}.
\qed

\section{Reduction argument}\label{Reduction}

Let $u$ be a solution of \eqref{propro}. For any $\varepsilon>0$, we define
\begin{equation*}
  v(x)=\varepsilon^{1/2}u(\varepsilon x).
\end{equation*}
Then $v$ solves the following problem
\begin{align}\label{propro'}
 \left\{
  \begin{array}{ll}
  -\Delta v=\displaystyle\Big(\int_{\Omega_\varepsilon}\frac{v^{6-\alpha}(y)}{|x-y|^\alpha}dy\Big)v^{5-\alpha}+\lambda\varepsilon^2 v,
  \ \  &\mbox{in}\ \Omega_\varepsilon,\\
  v=0,
  \ \  &\mbox{on}\ \partial \Omega_\varepsilon,
    \end{array}
    \right.
  \end{align}
where $\Omega_\varepsilon=\varepsilon^{-1}\Omega$. Define
\begin{equation*}
  \mathcal{I}_\lambda(v)=\frac{1}{2}\int_{\Omega_\varepsilon}|\nabla v|^2 dx-\frac{\lambda \varepsilon^2}{2}\int_{\Omega_\varepsilon} v^2 dx-\frac{1}{2(6-\alpha)}\int_{\Omega_\varepsilon}\int_{ \Omega_\varepsilon}\frac{v^{6-\alpha}(y)v^{6-\alpha}(x)}{|x-y|^\alpha}dxdy,
\end{equation*}
and
\begin{equation*}
V(x)=\varepsilon^{1/2}U_{\mu,\xi}(\varepsilon x)=w_{\mu',\xi'}(x)+\varepsilon^{1/2}\pi_{\mu,\xi}(\varepsilon x), \quad \mu'=\frac{\mu}{\varepsilon},\quad \xi'=\frac{\xi}{\varepsilon},\quad  \text{$x\in \Omega_\varepsilon$},
\end{equation*}
then $V$ satisfies
\begin{align}\label{propro''}
 \left\{
  \begin{array}{ll}
  -\Delta V=\displaystyle\Big(\int_{\Omega_\varepsilon}\frac{w_{\mu',\xi'}^{6-\alpha}(y)}{|x-y|^\alpha}dy\Big)w_{\mu',\xi'}^{5-\alpha}+\lambda\varepsilon^2 V,
  \ \  &\mbox{in}\ \Omega_\varepsilon,\\
  V=0,
  \ \  &\mbox{on}\ \partial \Omega_\varepsilon.
    \end{array}
    \right.
  \end{align}
Thus finding a solution of \eqref{propro} which is a small perturbation of $U_{\mu,\xi}$ is equivalent to finding a solution of \eqref{propro'} of the form:
\begin{equation*}
  V+\phi,
\end{equation*}
where $\phi$ is small in some  appropriate sense. This is equivalent to finding $\phi$ such that
\begin{align}\label{propro'''}
 \left\{
  \begin{array}{ll}
  L(\phi)=N(\phi)+E,
  \ \  &\mbox{in}\ \Omega_\varepsilon,\\
  \phi=0,
  \ \  &\mbox{on}\ \partial \Omega_\varepsilon,
    \end{array}
    \right.
  \end{align}
where
\begin{equation*}
  L(\phi)=-\Delta \phi -\lambda \varepsilon^2 \phi-(6-\alpha)\displaystyle\Big(\int_{\Omega_\varepsilon}\frac{V^{5-\alpha}(y)\phi(y)}{|x-y|^\alpha}dy\Big)V^{5-\alpha}-
  (5-\alpha)\displaystyle\Big(\int_{\Omega_\varepsilon}\frac{V^{6-\alpha}(y)}{|x-y|^\alpha}dy\Big)V^{4-\alpha}\phi,
\end{equation*}
\begin{align*}
  N(\phi)=&\displaystyle\Big(\int_{\Omega_\varepsilon}\frac{(V+\phi)^{6-\alpha}(y)}{|x-y|^\alpha}dy\Big)(V+\phi)^{5-\alpha}
  -\displaystyle\Big(\int_{\Omega_\varepsilon}\frac{V^{6-\alpha}(y)}{|x-y|^\alpha}dy\Big)V^{5-\alpha}\\
  &-(6-\alpha)\displaystyle\Big(\int_{\Omega_\varepsilon}\frac{V^{5-\alpha}(y)\phi(y)}{|x-y|^\alpha}dy\Big)V^{5-\alpha}-
  (5-\alpha)\displaystyle\Big(\int_{\Omega_\varepsilon}\frac{V^{6-\alpha}(y)}{|x-y|^\alpha}dy\Big)V^{4-\alpha}\phi,
\end{align*}
and
\begin{equation*}
  E=\displaystyle\Big(\int_{\Omega_\varepsilon}\frac{V^{6-\alpha}(y)}{|x-y|^\alpha}dy\Big)V^{5-\alpha}-\displaystyle\Big(\int_{\Omega_\varepsilon}\frac{w_{\mu',\xi'}^{6-\alpha}(y)}{|x-y|^\alpha}dy\Big)w_{\mu',\xi'}^{5-\alpha}.
\end{equation*}

By a direct computation, we have
\begin{equation}\label{gu2}
  \frac{\partial w_{\mu,\xi}}{\partial \mu}=\frac{3^{1/4}}{2}\frac{|x-\xi|^2-\mu^2}{\mu^{\frac{1}{2}}(\mu^2+|x-\xi|^2)^{\frac{3}{2}}}=O(\frac{w_{\mu,\xi}}{\mu}),
\end{equation}
and
\begin{equation}\label{gu1}
  \frac{\partial w_{\mu,\xi}}{\partial \xi_i}=-3^{1/4}\mu^{1/2}\frac{x_i-\xi_i}{(\mu^2+|x-\xi|^2)^{\frac{3}{2}}}=O(\frac{w_{\mu,\xi}}{\mu}),\quad \text{for $i=1,2,3$}.
\end{equation}
Moreover, by Lemma \ref{piex}, we have
\begin{equation}\label{gu3}
 \Big|\frac{\partial [\varepsilon^{1/2}\pi_{\mu,\xi}(\varepsilon x)]}{\partial \mu'}\Big|=O(\varepsilon)\quad \text{and}\quad
  \Big| \frac{\partial [\varepsilon^{1/2}\pi_{\mu,\xi}(\varepsilon x)]}{\partial \xi'_i}\Big|=O(\varepsilon^2),\quad \text{for $i=1,2,3$}.
\end{equation}
Then we have the following lemmas regarding $N(\phi)$ and $E$.
\begin{lemma}\label{non}
For any $\varepsilon>0$, if there exists $\delta>0$ such that
\begin{equation*}
  dist(\xi',\partial \Omega_\varepsilon)>\frac{\delta}{\varepsilon}\quad \text{and}\quad \mu'\in (\delta,\delta^{-1}),
\end{equation*}
then there holds
\begin{equation*}
  \|N(\phi)\|_{H_0^1(\Omega_\varepsilon)}\leq C\|\phi\|^2_{H_0^1(\Omega_\varepsilon)}.
\end{equation*}
\end{lemma}
\begin{proof}
For any $\varphi\in H_0^1(\Omega_\varepsilon)$, by the definition of $N(\phi)$, we have
\begin{align*}
  \Big|\int_{\Omega_\varepsilon}N(\phi)\varphi dx\Big|\leq& C\bigg|\int_{\Omega_\varepsilon}\int_{\Omega_\varepsilon}\frac{V^{6-\alpha}(y)V^{3-\alpha}(x)
  \phi^2(x)\varphi(x)}{|x-y|^\alpha}dxdy\bigg|\\
  &+C\bigg|\int_{\Omega_\varepsilon}\int_{\Omega_\varepsilon}\frac{V^{5-\alpha}(y)\phi(y)V^{4-\alpha}(x)
  \phi(x)\varphi(x)}{|x-y|^\alpha}dxdy\bigg|.
\end{align*}
Using \eqref{HLS}, the H\"{o}lder and Sobolev inequalities, we obtain
\begin{equation*}
 \bigg| \int_{\Omega_\varepsilon}\int_{\Omega_\varepsilon}\frac{V^{6-\alpha}(y)V^{3-\alpha}(x)
  \phi^2(x)\varphi(x)}{|x-y|^\alpha}dxdy\bigg|\leq C\Big(\int_{\Omega_\varepsilon}w_{\mu',\xi'}^{\frac{6(3-\alpha)}{6-\alpha}}\phi^{\frac{12}{6-\alpha}}\varphi^{\frac{6}{6-\alpha}}dx\Big)^{\frac{6-\alpha}{6}}\leq C
  \|\phi\|^2_{H_0^1(\Omega_\varepsilon)}\|\varphi\|_{H_0^1(\Omega_\varepsilon)},
\end{equation*}
and
\begin{align*}
  \bigg|\int_{\Omega_\varepsilon}\int_{\Omega_\varepsilon}\frac{V^{5-\alpha}(y)\phi(y)V^{4-\alpha}(x)
  \phi(x)\varphi(x)}{|x-y|^\alpha}dxdy\bigg|\leq & C\Big(\int_{\Omega_\varepsilon}w_{\mu',\xi'}^{\frac{6(5-\alpha)}{6-\alpha}}\phi^{\frac{6}{6-\alpha}}dx\Big)^{\frac{6-\alpha}{6}}
  \Big(\int_{\Omega_\varepsilon}w_{\mu',\xi'}^{\frac{6(4-\alpha)}{6-\alpha}}\phi^{\frac{6}{6-\alpha}}\varphi^{\frac{6}{6-\alpha}}dx\Big)^{\frac{6-\alpha}{6}}
\\
\leq &C
  \|\phi\|^2_{H_0^1(\Omega_\varepsilon)}\|\varphi\|_{H_0^1(\Omega_\varepsilon)}.
\end{align*}
This completes the proof.
\end{proof}

\begin{lemma}\label{err}
Under the conditions of Lemma \ref{non},
there holds
\begin{equation*}
  \|E\|_{H_0^1(\Omega_\varepsilon)}\leq C\varepsilon.
\end{equation*}
\end{lemma}
\begin{proof}
For any $\varphi\in H_0^1(\Omega_\varepsilon)$, we have
\begin{align*}
  &\Big|\int_{\Omega_\varepsilon}E\varphi dx\Big|\\
  =&\displaystyle \bigg|\int_{\Omega_\varepsilon}\int_{\Omega_\varepsilon}\frac{[V^{6-\alpha}(y)-w^{6-\alpha}_{\mu',\xi'}(y)]V^{5-\alpha}(x)\varphi(x)}{|x-y|^\alpha}dxdy-
\int_{\Omega_\varepsilon}\int_{\Omega_\varepsilon}\frac{w_{\mu',\xi'}^{6-\alpha}(y)[V^{5-\alpha}(x)-w^{5-\alpha}_{\mu',\xi'}(x)]\varphi(x)}{|x-y|^\alpha}dxdy\bigg|\\
  \\ \leq &C\bigg|\displaystyle \int_{\Omega_\varepsilon}\int_{\Omega_\varepsilon}\frac{w^{5-\alpha}_{\mu',\xi'}(y)\varepsilon^{1/2}\pi_{\mu,\xi}(\varepsilon y)V^{5-\alpha}(x)\varphi(x)}{|x-y|^\alpha}dxdy\bigg|+C\bigg|\displaystyle \int_{\Omega_\varepsilon}\int_{\Omega_\varepsilon}\frac{w^{6-\alpha}_{\mu',\xi'}(y)w_{\mu',\xi'}^{4-\alpha}(x)\varepsilon^{1/2}\pi_{\mu,\xi}(\varepsilon x)\varphi(x)}{|x-y|^\alpha}dxdy\bigg|.
\end{align*}
By Lemma \ref{piex}, using \eqref{HLS}, the H\"{o}lder and Sobolev inequalities, we deduce that
\begin{align*}
  \displaystyle \bigg|\int_{\Omega_\varepsilon}\int_{\Omega_\varepsilon}\frac{w^{6-\alpha}_{\mu',\xi'}(y)w_{\mu',\xi'}^{4-\alpha}(x)\varepsilon^{1/2}\pi_{\mu,\xi}(\varepsilon x)\varphi(x)}{|x-y|^\alpha}dxdy\bigg|&\leq C\varepsilon %\Big(\int_{\Omega_\varepsilon}w_{\mu',\xi'}^{6}dx\Big)^{\frac{6-\alpha}{6}}
  \Big(\int_{\Omega_\varepsilon}w_{\mu',\xi'}^{\frac{6(4-\alpha)}{6-\alpha}}\varphi^{\frac{6}{6-\alpha}}dx\Big)^{\frac{6-\alpha}{6}}\leq C\varepsilon\|\varphi\|_{H_0^1(\Omega_\varepsilon)}.
\end{align*}
Similarly, we can obtain
\begin{equation*}
  \displaystyle \bigg|\int_{\Omega_\varepsilon}\int_{\Omega_\varepsilon}\frac{w^{5-\alpha}_{\mu',\xi'}(y)\varepsilon^{1/2}\pi_{\mu,\xi}(\varepsilon y)V^{5-\alpha}(x)\varphi(x)}{|x-y|^\alpha}dxdy\bigg|\leq C\varepsilon\|\varphi\|_{H_0^1(\Omega_\varepsilon)}.
\end{equation*}
Hence the conclusion is reached.
\end{proof}

By Lemma \ref{ker},
we define
\begin{equation*}
  K_{\mu',\xi'}={\bf span}\Big\{\frac{\partial V}{\partial \xi'_1},\frac{\partial V}{\partial \xi'_2},\frac{\partial V}{\partial \xi'_3},\frac{\partial V}{\partial \mu'}\Big\},
\end{equation*}
and
\begin{equation*}
  K_{\mu',\xi'}^\perp=\Big\{\varphi\in H_0^1(\Omega_\varepsilon):\Big\langle\frac{\partial V}{\partial \mu'},\varphi\Big\rangle=0,\Big\langle\frac{\partial V}{\partial \xi'_i},\varphi\Big\rangle=0,\,\, \text{for $i=1,2,3$}\Big\},
\end{equation*}
where $\langle\cdot,\cdot\rangle$ denotes the inner product in the Sobolev space $H_0^1(\Omega_\varepsilon)$. Then we define the projections $\Pi_{\mu',\xi'}$ and $\Pi^\perp_{\mu',\xi'}$ of the Sobolev space $H_0^1(\Omega_\varepsilon)$ onto $K_{\mu',\xi'}$ and $K_{\mu',\xi'}^\perp$ respectively. We first
solve the following problem
\begin{equation}\label{zhuan}
  \Pi^\perp_{\mu',\xi'}L(\phi)=\Pi^\perp_{\mu',\xi'}(N(\phi)+E),
\end{equation}
and we have the following lemma.
\begin{proposition}\label{fixed}
Under the conditions of Lemma \ref{non},
equation \eqref{zhuan} admits a unique solution $\phi_{\mu',\xi'}$ in $K^\perp_{\mu',\xi'}$, which is continuously differentiable with respect to $\mu'$ and $\xi'$, such that
\begin{equation*}
  \|\phi_{\mu',\xi'}\|_{H_0^1(\Omega_\varepsilon)}\leq C\varepsilon.
\end{equation*}
\end{proposition}
For the proof of Proposition \ref{fixed}, we need the following lemma.
\begin{lemma}\label{xian}
Under the conditions of Lemma \ref{non}, for any $\varepsilon>0$, there exists a constant $\varrho>0$ such that %for any $\varepsilon>0$ small enough
\begin{equation*}
  \|\Pi_{\mu',\xi'}^\perp L(\phi)\|_{H_0^1(\Omega_\varepsilon)}\geq \varrho \|\phi\|_{H_0^1(\Omega_\varepsilon)},\quad \forall \,\, \phi \in K_{\mu',\xi'}^\perp.
\end{equation*}
\end{lemma}
\begin{proof}
We adopt the idea of \cite[Lemma 3.4]{YZ} to complete our proof. Assume by contradiction that there exist $\varepsilon_n\rightarrow0$ as $n\rightarrow\infty$, $\xi'_n\in \Omega_\varepsilon$ with $dist(\xi'_n,\partial \Omega_\varepsilon)>\frac{\delta}{\varepsilon}$, $\mu'_n,\lambda_n\in (\delta,\delta^{-1})$, and $\phi_n\in K_{\mu'_n,\xi'_n}^\perp$ such that
\begin{equation*}
  \|\Pi_{\mu'_n,\xi'_n}^\perp L(\phi_n)\|_{H_0^1(\Omega_\varepsilon)}\leq \frac{1}{n} \|\phi_n\|_{H_0^1(\Omega_\varepsilon)}.
\end{equation*}
We may assume that $\|\phi_n\|_{H_0^1(\Omega_\varepsilon)}=1$. Then for any $\varphi\in K_{\mu'_n,\xi'_n}^\perp$, we have
\begin{align}\label{1}
  \int_{\Omega_\varepsilon}\nabla \phi_n \cdot \nabla\varphi dx-\lambda_n \varepsilon_n^2\int_{\Omega_\varepsilon}\phi_n \varphi dx&-(6-\alpha)
  \displaystyle \int_{\Omega_\varepsilon}\int_{\Omega_\varepsilon}\frac{V_n^{5-\alpha}(y)\phi_n(y)V_n^{5-\alpha}(x)\varphi(x)}{|x-y|^\alpha}dxdy
 \nonumber \\&-(5-\alpha)
  \displaystyle \int_{\Omega_\varepsilon}\int_{\Omega_\varepsilon}\frac{V_n^{6-\alpha}(y)V_n^{4-\alpha}(x)\phi_n(x)\varphi(x)}{|x-y|^\alpha}dxdy \nonumber\\
  =&\langle L(\phi_n),\varphi\rangle=\big\langle \Pi_{\mu'_n,\xi'_n}^\perp L(\phi_n),\varphi\big\rangle\leq o(1)\|\varphi\|_{H_0^1(\Omega_\varepsilon)}.
\end{align}
Let $\varphi=\phi_n$, we find
\begin{align}\label{contra}
  \int_{\Omega_\varepsilon}|\nabla \phi_n|^2 dx&-(6-\alpha)
  \displaystyle \int_{\Omega_\varepsilon}\int_{\Omega_\varepsilon}\frac{V_n^{5-\alpha}(y)\phi_n(y)V_n^{5-\alpha}(x)\phi_n(x)}{|x-y|^\alpha}dxdy
  \nonumber\\&-(5-\alpha)
  \displaystyle \int_{\Omega_\varepsilon}\int_{\Omega_\varepsilon}\frac{V_n^{6-\alpha}(y)V_n^{4-\alpha}(x)\phi_n^2(x)}{|x-y|^\alpha}dxdy=o(1).
\end{align}
Next, we define $\tilde{\phi}_n(x)=\phi_n(x+\xi'_n)$. Then $\displaystyle\int_{\mathbb{R}^3}|\nabla \tilde{\phi}_n|^2dx\leq C$ and $\tilde{\phi}_n\in K_{\mu'_n,0}^\perp$. Up to a subsequence, we assume that $\tilde{\phi}_n\rightharpoonup \tilde{\phi}$ in $D^{1,2}(\mathbb{R}^3)$. From \eqref{1}, we expect that $\tilde{\phi}$ satisfies
\begin{equation}\label{2}
  \displaystyle -\Delta \tilde{\phi}-(6-\alpha)\Big(\int_{\mathbb{R}^3}\frac{w_{\mu',0}^{5-\alpha}(y)\tilde{\phi}(y)}{|x-y|^\alpha}dy\Big)w_{\mu',0}^{5-\alpha}-
 (5-\alpha)\Big(\int_{\mathbb{R}^3}\frac{w_{\mu',0}^{6-\alpha}(y)}{|x-y|^\alpha}dy\Big)w_{\mu',0}^{4-\alpha}\tilde{\phi}=0.
\end{equation}
The major difficulty to prove this claim is that \eqref{1} holds just for $\varphi\in K_{\mu'_n,\xi'_n}^\perp$, not for all $\varphi\in D^{1,2}(\mathbb{R}^3)$.

Now, we give the proof of \eqref{2}. For any $\varphi\in D^{1,2}(\mathbb{R}^3)$, there exist some constants $c_{\varepsilon_n,0}$ and $c_{\varepsilon_n,j}$ ($j=1,2,3$) such that
\begin{equation*}
 \varphi- \Pi_{\mu'_n,\xi'_n}^\perp \varphi=c_{\varepsilon_n,0} \frac{\partial V_n}{\partial \mu'_n}+\sum\limits_{j=1}^3 c_{\varepsilon_n,j} \frac{\partial V_n}{\partial \xi'_{n,j}}.
\end{equation*}
Since $ \big\langle\frac{\partial V_n}{\partial \mu'_n},\Pi_{\mu'_n,\xi'_n}^\perp \varphi\big\rangle=0$ and $\big\langle\frac{\partial V_n}{\partial \xi'_{n,i}},\Pi_{\mu'_n,\xi'_n}^\perp \varphi \big\rangle=0$ for $i=1,2,3$, we have
\begin{equation*}
  \Big\langle\frac{\partial V_n}{\partial \mu'_n},\varphi\Big\rangle=c_{\varepsilon_n,0}\Big\langle\frac{\partial V_n}{\partial \mu'_n},\frac{\partial V_n}{\partial \mu'_n}\Big\rangle\quad \text{and}\quad \Big\langle\frac{\partial V_n}{\partial \xi'_{n,i}},\varphi\Big\rangle=\delta_{ij}c_{\varepsilon_n,j}\Big\langle\frac{\partial V_n}{\partial \mu'_{n,i}},\frac{\partial V_n}{\partial \mu'_{n,j}}\Big\rangle.
\end{equation*}
Thus
\begin{equation*}
  c_{\varepsilon_n,0}=a_n\Big\langle\frac{\partial V_n}{\partial \mu'_n},\varphi\Big\rangle\quad \text{and}\quad c_{\varepsilon_n,j}=b_{n,j}\Big\langle\frac{\partial V_n}{\partial \xi'_{n,j}},\varphi\Big\rangle
\end{equation*}
for some constants $a_n$ and $b_{n,j}$, $j=1,2,3$. Hence, we obtain
\begin{align*}
  \int_{\Omega_\varepsilon}\nabla \phi_n \cdot \nabla\varphi dx-\lambda_n \varepsilon_n^2\int_{\Omega_\varepsilon}\phi_n \varphi dx&-(6-\alpha)
  \displaystyle \int_{\Omega_\varepsilon}\int_{\Omega_\varepsilon}\frac{V_n^{5-\alpha}(y)\phi_n(y)V_n^{5-\alpha}(x)\varphi(x)}{|x-y|^\alpha}dxdy
 \nonumber \\&-(5-\alpha)
  \displaystyle \int_{\Omega_\varepsilon}\int_{\Omega_\varepsilon}\frac{V_n^{6-\alpha}(y)V_n^{4-\alpha}(x)\phi_n(x)\varphi(x)}{|x-y|^\alpha}dxdy \nonumber\\
  =&\langle L(\phi_n),\varphi\rangle \nonumber\\
  =&\big\langle L(\phi_n),\Pi_{\mu'_n,\xi'_n}^\perp \varphi\big\rangle+c_{\varepsilon_n,0} \Big\langle L(\phi_n),\frac{\partial V_n}{\partial \mu'_n}\Big\rangle+\sum\limits_{j=1}^3 c_{\varepsilon_n,j} \Big\langle L(\phi_n),\frac{\partial V_n}{\partial \xi'_{n,j}}\Big\rangle.
\end{align*}
Observe that
\begin{equation*}
  \big|\big\langle L(\phi_n),\Pi_{\mu'_n,\xi'_n}^\perp \varphi\big\rangle\big|\leq o(1)\|\Pi_{\mu'_n,\xi'_n}^\perp \varphi\|_{H_0^1(\Omega_\varepsilon)}\leq
  o(1)\|\varphi\|_{H_0^1(\Omega_\varepsilon)},
\end{equation*}
and
\begin{equation*}
  \Big|\lambda_n \varepsilon_n^2\int_{\Omega_\varepsilon}\phi_n \varphi dx\Big|\leq \lambda_n \varepsilon_n^2\|\varphi\|_{H_0^1(\Omega_\varepsilon)}=o(1)\|\varphi\|_{H_0^1(\Omega_\varepsilon)},
\end{equation*}
we obtain
\begin{align}\label{3}
  \int_{\Omega_\varepsilon}\nabla \phi_n \cdot \nabla\varphi dx-&(6-\alpha)
  \displaystyle \int_{\Omega_\varepsilon}\int_{\Omega_\varepsilon}\frac{V_n^{5-\alpha}(y)\phi_n(y)V_n^{5-\alpha}(x)\varphi(x)}{|x-y|^\alpha}dxdy
 \nonumber \\&-(5-\alpha)
  \displaystyle \int_{\Omega_\varepsilon}\int_{\Omega_\varepsilon}\frac{V_n^{6-\alpha}(y)V_n^{4-\alpha}(x)\phi_n(x)\varphi(x)}{|x-y|^\alpha}dxdy \nonumber\\
  =&o(1)\|\varphi\|_{H_0^1(\Omega_\varepsilon)}+\tilde{a}_{n} \Big\langle \frac{\partial V_n}{\partial \mu'_n},\varphi\Big\rangle+\sum\limits_{j=1}^3 \tilde{b}_{n,j} \Big\langle \frac{\partial V_n}{\partial \xi'_{n,j}},\varphi\Big\rangle
\end{align}
for some constants $\tilde{a}_n$ and $\tilde{b}_{n,j}$, $j=1,2,3$. In the following, we prove that
\begin{equation}\label{4}
  \Big|\tilde{a}_{n} \Big\langle \frac{\partial V_n}{\partial \mu'_n},\varphi\Big\rangle\Big|=o(1)\|\varphi\|_{H_0^1(\Omega_\varepsilon)}\quad \text{and}\quad\Big|\tilde{b}_{n,j}\Big \langle \frac{\partial V_n}{\partial \xi'_{n,j}},\varphi\Big\rangle\Big|=o(1)\|\varphi\|_{H_0^1(\Omega_\varepsilon)},\quad \text{for $j=1,2,3$}.
\end{equation}

Taking $\varphi=\frac{\partial V_n}{\partial \xi'_{n,k}}$ in \eqref{3}, $1\leq k \leq 3$, since $\phi_n\in K_{\mu'_n,\xi'_n}^\perp$, we obtain
\begin{align}\label{com1}
  &\sum\limits_{j=1}^3 \tilde{b}_{n,j} \Big\langle \frac{\partial V_n}{\partial \xi'_{n,j}},\frac{\partial V_n}{\partial \xi'_{n,k}}\Big\rangle \nonumber\\
  =&-(6-\alpha)
  \displaystyle \int_{\Omega_\varepsilon}\int_{\Omega_\varepsilon}\frac{V_n^{5-\alpha}(y)\phi_n(y)V_n^{5-\alpha}(x)\frac{\partial V_n}{\partial \xi'_{n,k}}(x)}{|x-y|^\alpha}dxdy
\nonumber \\&-(5-\alpha)
  \displaystyle \int_{\Omega_\varepsilon}\int_{\Omega_\varepsilon}\frac{V_n^{6-\alpha}(y)V_n^{4-\alpha}(x)\phi_n(x)\frac{\partial V_n}{\partial \xi'_{n,k}}(x)}{|x-y|^\alpha}dxdy+ o(1)\Big\|\frac{\partial V_n}{\partial \xi'_{n,k}}\Big\|_{H_0^1(\Omega_\varepsilon)} \nonumber\\
  =&:-(6-\alpha)A_1-(5-\alpha)A_2+ o(1)\Big\|\frac{\partial V_n}{\partial \xi'_{n,k}}\Big\|_{H_0^1(\Omega_\varepsilon)}.
\end{align}
From \eqref{propro''}, it follows
\begin{align}\label{5}
  -\Delta \frac{\partial V_n}{\partial {\xi'_{n,k}}}=&\displaystyle(6-\alpha)\Big(\int_{\Omega_\varepsilon}\frac{w_{\mu'_n,\xi'_n}^{5-\alpha}(y)\frac{\partial w_{\mu'_n,\xi'_n}}{\partial \xi'_{n,k} }(y)}{|x-y|^\alpha}dy\Big)w_{\mu'_n,\xi'_n}^{5-\alpha} \nonumber\\
  &+(5-\alpha)\Big(\int_{\Omega_\varepsilon}\frac{w_{\mu'_n,\xi'_n}^{6-\alpha}(y)}{|x-y|^\alpha}dy\Big)w_{\mu'_n,\xi'_n}^{4-\alpha}\frac{\partial w_{\mu'_n,\xi'_n}}{\partial \xi'_{n,k} }+\lambda_n\varepsilon_n^2 \frac{\partial V_n}{\partial {\xi'_{n,k}}}.
  \end{align}
Using \eqref{gu1}, we get
\begin{align*}
  \int_{\Omega_\varepsilon}\Big|\nabla \frac{\partial V_n}{\partial {\xi'_{n,k}}}\Big|^2dx=&\displaystyle(6-\alpha)\int_{\Omega_\varepsilon}\int_{\Omega_\varepsilon}\frac{w_{\mu'_n,\xi'_n}^{5-\alpha}(y)\frac{\partial w_{\mu'_n,\xi'_n}}{\partial \xi'_{n,k} }(y)w_{\mu'_n,\xi'_n}^{5-\alpha}(x)\frac{\partial V_n}{\partial {\xi'_{n,k}}}(x)}{|x-y|^\alpha}dy\\
  &+(5-\alpha)\int_{\Omega_\varepsilon}\int_{\Omega_\varepsilon}\frac{w_{\mu'_n,\xi'_n}^{6-\alpha}(y)w_{\mu'_n,\xi'_n}^{4-\alpha}(x)\frac{\partial w_{\mu'_n,\xi'_n}}{\partial \xi'_{n,k} }(x)\frac{\partial V_n}{\partial {\xi'_{n,k}}}(x)}{|x-y|^\alpha}dy+\lambda_n\varepsilon_n^2\int_{\Omega_\varepsilon}
   \Big|\frac{\partial V_n}{\partial {\xi'_{n,k}}}\Big|^2dx\\
   =:&(6-\alpha)A_3+(5-\alpha)A_4+O(\varepsilon_n).
  \end{align*}
By a direct computation, we obtain $A_3=O(1)$ and $A_4=O(1)$.
Repeating the above estimate, we can also find
\begin{equation}\label{com2}
 \Big \langle \frac{\partial V_n}{\partial \xi'_{n,j}},\frac{\partial V_n}{\partial \xi'_{n,k}}\Big\rangle=O(1).
\end{equation}
On the other hand, by $\phi_n\in K_{\mu'_n,\xi'_n}^\perp$, \eqref{gu1} and \eqref{5}, we have
\begin{align*}
  (6-\alpha)A_1+(5-\alpha)A_2=&(6-\alpha)
  \displaystyle \int_{\Omega_\varepsilon}\int_{\Omega_\varepsilon}\frac{V_n^{5-\alpha}(y)\phi_n(y)V_n^{5-\alpha}(x)\frac{\partial V_n}{\partial \xi'_{n,k}}(x)}{|x-y|^\alpha}dxdy
 \\
 &+(5-\alpha)
  \displaystyle \int_{\Omega_\varepsilon}\int_{\Omega_\varepsilon}\frac{V_n^{6-\alpha}(y)V_n^{4-\alpha}(x)\phi_n(x)\frac{\partial V_n}{\partial \xi'_{n,k}}(x)}{|x-y|^\alpha}dxdy\\
  =&(6-\alpha)
  \displaystyle \int_{\Omega_\varepsilon}\int_{\Omega_\varepsilon}\frac{V_n^{5-\alpha}(y)\phi_n(y)V_n^{5-\alpha}(x)\frac{\partial w_{\mu'_n,\xi'_n}}{\partial \xi'_{n,k}}(x)}{|x-y|^\alpha}dxdy
 \\
 &+(5-\alpha)
  \displaystyle \int_{\Omega_\varepsilon}\int_{\Omega_\varepsilon}\frac{V_n^{6-\alpha}(y)V_n^{4-\alpha}(x)\phi_n(x)\frac{\partial w_{\mu'_n,\xi'_n}}{\partial \xi'_{n,k}}(x)}{|x-y|^\alpha}dxdy\\
  &+(6-\alpha)
  \displaystyle \int_{\Omega_\varepsilon}\int_{\Omega_\varepsilon}\frac{V_n^{5-\alpha}(y)\phi_n(y)V_n^{5-\alpha}(x)\frac{\partial [\varepsilon_n^{1/2}\pi_{\mu_n,\xi_n}(\varepsilon_n x)]}{\partial \xi'_{n,k}}(x)}{|x-y|^\alpha}dxdy
 \\
 &+(5-\alpha)
  \displaystyle \int_{\Omega_\varepsilon}\int_{\Omega_\varepsilon}\frac{V_n^{6-\alpha}(y)V_n^{4-\alpha}(x)\phi_n(x)\frac{\partial [\varepsilon_n^{1/2}\pi_{\mu_n,\xi_n}(\varepsilon_n x)]}{\partial \xi'_{n,k}}(x)}{|x-y|^\alpha}dxdy\\
  %=&P+(6-\alpha)
  %\displaystyle \int_{\Omega_\varepsilon}\int_{\Omega_\varepsilon}\frac{w_{\mu'_n,\xi'_n}^{5-\alpha}(y)\phi_n(y)w_{\mu'_n,\xi'_n}^{5-\alpha}(x)\frac{\partial w_{\mu'_n,\xi'_n}}{\partial \xi'_{n,k}}(x)}{|x-y|^\alpha}dxdy
 %\\
 %&+(5-\alpha)
 % \displaystyle \int_{\Omega_\varepsilon}\int_{\Omega_\varepsilon}\frac{w_{\mu'_n,\xi'_n}^{6-\alpha}(y)w_{\mu'_n,\xi'_n}^{4-\alpha}(x)\phi_n(x)\frac{\partial w_{\mu'_n,\xi'_n}}{\partial \xi'_{n,k}}(x)}{|x-y|^\alpha}dxdy
 % \\
 % &+(6-\alpha)
 % \displaystyle \int_{\Omega_\varepsilon}\int_{\Omega_\varepsilon}\frac{V_n^{5-\alpha}(y)\phi_n(y)V_n^{5-\alpha}(x)\frac{\partial [\varepsilon_n^{1/2}\pi_{\mu_n,\xi_n}(\varepsilon_n x)]}{\partial \xi'_{n,k}}(x)}{|x-y|^\alpha}dxdy
 %\\
 %&+(5-\alpha)
 % \displaystyle \int_{\Omega_\varepsilon}\int_{\Omega_\varepsilon}\frac{V_n^{6-\alpha}(y)V_n^{4-\alpha}(x)\phi_n(x)\frac{\partial [\varepsilon_n^{1/2}\pi_{\mu_n,\xi_n}(\varepsilon_n x)]}{\partial \xi'_{n,k}}(x)}{|x-y|^\alpha}dxdy\\
  =&P+(6-\alpha)
  \displaystyle \int_{\Omega_\varepsilon}\int_{\Omega_\varepsilon}\frac{V_n^{5-\alpha}(y)\phi_n(y)V_n^{5-\alpha}(x)\frac{\partial [\varepsilon_n^{1/2}\pi_{\mu_n,\xi_n}(\varepsilon_n x)]}{\partial \xi'_{n,k}}(x)}{|x-y|^\alpha}dxdy
 \\
 &+(5-\alpha)
  \displaystyle \int_{\Omega_\varepsilon}\int_{\Omega_\varepsilon}\frac{V_n^{6-\alpha}(y)V_n^{4-\alpha}(x)\phi_n(x)\frac{\partial [\varepsilon_n^{1/2}\pi_{\mu_n,\xi_n}(\varepsilon_n x)]}{\partial \xi'_{n,k}}(x)}{|x-y|^\alpha}dxdy\\
  =:&P+(6-\alpha)A_5+(5-\alpha)A_6,
\end{align*}
where
\begin{align*}
  |P|=&\bigg|-\lambda_n\varepsilon_n^2\int_{\Omega_\varepsilon} \frac{\partial V_n}{\partial {\xi'_{n,k}}}\phi_ndx+(6-\alpha)
  \displaystyle \int_{\Omega_\varepsilon}\int_{\Omega_\varepsilon}\frac{V_n^{5-\alpha}(y)\phi_n(y)V_n^{5-\alpha}(x)\frac{\partial w_{\mu'_n,\xi'_n}}{\partial \xi'_{n,k}}(x)}{|x-y|^\alpha}dxdy
 \\
 &+(5-\alpha)
  \displaystyle \int_{\Omega_\varepsilon}\int_{\Omega_\varepsilon}\frac{V_n^{6-\alpha}(y)V_n^{4-\alpha}(x)\phi_n(x)\frac{\partial w_{\mu'_n,\xi'_n}}{\partial \xi'_{n,k}}(x)}{|x-y|^\alpha}dxdy\\
  &-(6-\alpha)
  \displaystyle \int_{\Omega_\varepsilon}\int_{\Omega_\varepsilon}\frac{w_{\mu'_n,\xi'_n}^{5-\alpha}(y)\phi_n(y)w_{\mu'_n,\xi'_n}^{5-\alpha}(x)\frac{\partial w_{\mu'_n,\xi'_n}}{\partial \xi'_{n,k}}(x)}{|x-y|^\alpha}dxdy
 \\
 &-(5-\alpha)
  \displaystyle \int_{\Omega_\varepsilon}\int_{\Omega_\varepsilon}\frac{w_{\mu'_n,\xi'_n}^{6-\alpha}(y)w_{\mu'_n,\xi'_n}^{4-\alpha}(x)\phi_n(x)\frac{\partial w_{\mu'_n,\xi'_n}}{\partial \xi'_{n,k}}(x)}{|x-y|^\alpha}dxdy\bigg|\\
  \leq &C\lambda_n\varepsilon_n^2\Big \|\frac{\partial V_n}{\partial {\xi'_{n,k}}}\Big\|_{H_0^1(\Omega_\varepsilon)} + C\bigg|\int_{\Omega_\varepsilon}\int_{\Omega_\varepsilon}\frac{w_{\mu'_n,\xi'_n}^{4-\alpha}(y)\varepsilon_n^{1/2}\pi_{\mu_n,\xi_n}(\varepsilon_n y)\phi_n(y)w_{\mu'_n,\xi'_n}^{5-\alpha}(x)\frac{\partial w_{\mu'_n,\xi'_n}}{\partial \xi'_{n,k}}(x)}{|x-y|^\alpha}dxdy\bigg|\\
  &+C\bigg| \int_{\Omega_\varepsilon}\int_{\Omega_\varepsilon}\frac{w_{\mu'_n,\xi'_n}^{5-\alpha}(y)\varepsilon_n^{1/2}\pi_{\mu_n,\xi_n}(\varepsilon_n y)w_{\mu'_n,\xi'_n}^{4-\alpha}(x)\phi_n(x)\frac{\partial w_{\mu'_n,\xi'_n}}{\partial \xi'_{n,k}}(x)}{|x-y|^\alpha}dxdy\bigg|\leq C\varepsilon_n.
\end{align*}
Moreover, a direct computation shows that $A_5=O(\varepsilon_n^2)$ and $A_6=O(\varepsilon_n^2)$. Hence, we have
\begin{equation*}
  (6-\alpha)A_1+(5-\alpha)A_2=O(\varepsilon_n).
\end{equation*}
This with \eqref{com1} and \eqref{com2} yields that $|\tilde{b}_{n,j}|=o(1)$ for $j=1,2,3$.
Therefore, we can deduce
\begin{equation*}
 \Big| \tilde{b}_{n,j} \Big\langle \frac{\partial V_n}{\partial \xi'_{n,j}},\varphi\Big\rangle \Big|\leq |\tilde{b}_{n,j}| \times \Big\|\frac{\partial V_n}{\partial \xi'_{n,j}}\Big\|_{H_0^1(\Omega_\varepsilon)}\|\varphi\|_{H_0^1(\Omega_\varepsilon)}= o(1)\|\varphi\|_{H_0^1(\Omega_\varepsilon)}.
\end{equation*}
Similarly, taking $\varphi=\frac{\partial V_n}{\partial \mu'_{n}}$ in \eqref{3}, we can prove that $|\tilde{a}_{n}|=o(1)$. Then we obtain
\begin{equation*}
 \Big| \tilde{a}_{n} \Big\langle \frac{\partial V_n}{\partial \mu'_{n}},\varphi\Big\rangle \Big|\leq |\tilde{a}_{n}| \times\Big\|\frac{\partial V_n}{\partial \mu'_{n}}\Big\|_{H_0^1(\Omega_\varepsilon)}\|\varphi\|_{H_0^1(\Omega_\varepsilon)}= o(1)\|\varphi\|_{H_0^1(\Omega_\varepsilon)}.
\end{equation*}
This completes the proof of \eqref{4}. Consequently, \eqref{3} becomes
\begin{align}\label{become}
  &\int_{\Omega_\varepsilon}\nabla \phi_n \cdot \nabla\varphi dx-(6-\alpha)
  \displaystyle \int_{\Omega_\varepsilon}\int_{\Omega_\varepsilon}\frac{V_n^{5-\alpha}(y)\phi_n(y)V_n^{5-\alpha}(x)\varphi(x)}{|x-y|^\alpha}dxdy
 \nonumber \\&-(5-\alpha)
  \displaystyle \int_{\Omega_\varepsilon}\int_{\Omega_\varepsilon}\frac{V_n^{6-\alpha}(y)V_n^{4-\alpha}(x)\phi_n(x)\varphi(x)}{|x-y|^\alpha}dxdy
  =o(1)\|\varphi\|_{H_0^1(\Omega_\varepsilon)}.
\end{align}

Next, for any $\varphi \in D^{1,2}(\mathbb{R}^3)$, let $\tilde{\varphi}_n(x)=\varphi(x-\xi'_n)$. Then from \eqref{become}, we have
\begin{align*}
  &\int_{\Omega_\varepsilon+\xi'_n}\nabla\tilde{ \phi}_n \cdot \nabla\varphi dx-(6-\alpha)
  \displaystyle \int_{\Omega_\varepsilon+\xi'_n}\int_{\Omega_\varepsilon+\xi'_n}\frac{w_{\mu'_n,0}^{5-\alpha}(y)\tilde{\phi}_n(y)w_{\mu'_n,0}^{5-\alpha}(x){\varphi}(x)}
  {|x-y|^\alpha}dxdy
 \nonumber \\&-(5-\alpha)
  \displaystyle \int_{\Omega_\varepsilon+\xi'_n}\int_{\Omega_\varepsilon+\xi'_n}\frac{w_{\mu'_n,0}^{6-\alpha}(y)w_{\mu'_n,0}^{4-\alpha}(x)\tilde{\phi}_n(x)\varphi(x)}{|x-y|^\alpha}dxdy
  \\
  =&\int_{\Omega_\varepsilon}\nabla{ \phi}_n \cdot \nabla\tilde{\varphi}_n dx-(6-\alpha)
  \displaystyle \int_{\Omega_\varepsilon}\int_{\Omega_\varepsilon}\frac{w_{\mu'_n,\xi'_n}^{5-\alpha}(y){\phi}_n(y)w_{\mu'_n,\xi'_n}^{5-\alpha}(x)\tilde{{\varphi}}_n(x)}
  {|x-y|^\alpha}dxdy
 \nonumber \\&-(5-\alpha)
  \displaystyle \int_{\Omega_\varepsilon}\int_{\Omega_\varepsilon}\frac{w_{\mu'_n,\xi'_n}^{6-\alpha}(y)w_{\mu'_n,\xi'_n}^{4-\alpha}(x){\phi}_n(x)\tilde{\varphi}_n(x)}{|x-y|^\alpha}dxdy
  =o(1)\|\tilde{\varphi}_n\|_{H_0^1(\Omega_\varepsilon)}=o(1)\|\varphi\|_{D^{1,2}(\mathbb{R}^3)}.
\end{align*}
Taking the limit as $n\rightarrow+\infty$, then $\tilde{\phi}$ satisfies
\begin{equation*}
  \displaystyle -\Delta \tilde{\phi}-(6-\alpha)\Big(\int_{\mathbb{R}^3}\frac{w_{\mu',0}^{5-\alpha}(y)\tilde{\phi}(y)}{|x-y|^\alpha}dy\Big)w_{\mu',0}^{5-\alpha}-
 (5-\alpha)\Big(\int_{\mathbb{R}^3}\frac{w_{\mu',0}^{6-\alpha}(y)}{|x-y|^\alpha}dy\Big)w_{\mu',0}^{4-\alpha}\tilde{\phi}=0.
\end{equation*}
This proves \eqref{2}.
From the non-degeneracy of solution $w_{\mu',0}$ and $\tilde{\phi}_n\in K_{\mu'_n,0}^\perp$, we obtain $\tilde{\phi}=0$.
Using \eqref{HLS}, the H\"{o}lder and Sobolev inequalities, we obtain
\begin{align*}
  &(6-\alpha)
  \displaystyle \int_{\Omega_\varepsilon}\int_{\Omega_\varepsilon}\frac{V_n^{5-\alpha}(y)\phi_n(y)V_n^{5-\alpha}(x)\phi_n(x)}{|x-y|^\alpha}dxdy
  +(5-\alpha)
  \displaystyle \int_{\Omega_\varepsilon}\int_{\Omega_\varepsilon}\frac{V_n^{6-\alpha}(y)V_n^{4-\alpha}(x)\phi_n^2(x)}{|x-y|^\alpha}dxdy\\
  \leq &C \Big(\int_{\Omega_\varepsilon}w_{\mu'_n,\xi'_n}^{\frac{6(5-\alpha)}{6-\alpha}}\phi_n^{\frac{6}{6-\alpha}}dx\Big)^{\frac{6-\alpha}{3}}+
   C \Big(\int_{\Omega_\varepsilon}w_{\mu'_n,\xi'_n}^{\frac{6(4-\alpha)}{6-\alpha}}\phi_n^{\frac{12}{6-\alpha}}dx\Big)^{\frac{6-\alpha}{6}}\\
   \leq &C\|\phi_n\|_{L^6(\Omega_\varepsilon)}^2=o(1).
\end{align*}
Then it follows from \eqref{contra} that
\begin{equation*}
  \int_{\Omega_\varepsilon}|\nabla \phi_n|^2 dx=o(1),
\end{equation*}
which is a contradiction. Thus we finish the proof of Lemma \ref{xian}.
\end{proof}
\noindent{\bf Proof of Proposition \ref{fixed}.}
By Lemma \ref{xian}, we can rewrite \eqref{zhuan} as
\begin{equation}\label{map}
  \phi=T(\phi):=(\Pi^\perp_{\mu',\xi'}L)^{-1}\big(\Pi^\perp_{\mu',\xi'}(N(\phi)+E)\big).
\end{equation}
We define a ball
\begin{equation*}
  \mathcal{B}:=\Big\{\phi\in K^\perp_{\mu',\xi'}:\|\phi\|_{H_0^1(\Omega_\varepsilon)}\leq C\varepsilon\Big\}.
\end{equation*}
In the following, we prove that $T$ maps $\mathcal{B}$ to $\mathcal{B}$ and $T$ is a contraction map. Hence, $T$ admits a fixed point $\phi_{\mu',\xi'}\in \mathcal{B}$.

First, by Lemmas \ref{non}-\ref{xian}, for any $\phi\in \mathcal{B}$, we have
\begin{equation*}
  \|T(\phi)\|_{H_0^1(\Omega_\varepsilon)}\leq C\|N(\phi)+E\|_{H_0^1(\Omega_\varepsilon)}\leq C\|N(\phi)\|_{H_0^1(\Omega_\varepsilon)}+C\|E\|_{H_0^1(\Omega_\varepsilon)}\leq C\varepsilon.
\end{equation*}

Second, for any $\phi_1,\phi_2\in \mathcal{B}$, we have
\begin{equation*}
  \|T(\phi_1)-T(\phi_2)\|_{H_0^1(\Omega_\varepsilon)}\leq C\|N(\phi_1)-N(\phi_2)\|_{H_0^1(\Omega_\varepsilon)}.
\end{equation*}
On the other hand, we know
\begin{align*}
  N(\phi_1)-N(\phi_2)=&\Big(\int_{\Omega_\varepsilon}\frac{(V+\phi_1)^{6-\alpha}(y)}{|x-y|^\alpha}dy\Big)(V+\phi_1)^{5-\alpha}
  -\Big(\int_{\Omega_\varepsilon}\frac{(V+\phi_2)^{6-\alpha}(y)}{|x-y|^\alpha}dy\Big)(V+\phi_2)^{5-\alpha}\\
  &-(6-\alpha)\displaystyle\Big(\int_{\Omega_\varepsilon}\frac{V^{5-\alpha}(y)(\phi_1-\phi_2)}{|x-y|^\alpha}dy\Big)V^{5-\alpha}-
  (5-\alpha)\displaystyle\Big(\int_{\Omega_\varepsilon}\frac{V^{6-\alpha}(y)}{|x-y|^\alpha}dy\Big)V^{4-\alpha}(\phi_1-\phi_2)\\
  =&\Big(\int_{\Omega_\varepsilon}\frac{(V+\phi_1)^{6-\alpha}(y)}{|x-y|^\alpha}dy\Big)\big[(V+\phi_1)^{5-\alpha}-(V+\phi_2)^{5-\alpha}\big]\\
  &-(5-\alpha)\Big(\int_{\Omega_\varepsilon}\frac{(V+\phi_1)^{6-\alpha}(y)}{|x-y|^\alpha}dy\Big)\big[V+\phi_1+\vartheta(\phi_1-\phi_2)\big]^{4-\alpha}(\phi_1-\phi_2)
  \\&+\Big(\int_{\Omega_\varepsilon}\frac{\big[(V+\phi_1)^{6-\alpha}-(V+\phi_2)^{6-\alpha}\big]}{|x-y|^\alpha}dy\Big)(V+\phi_2)^{5-\alpha}\\
  &-(6-\alpha)\Big(\int_{\Omega_\varepsilon}\frac{\big[V+\phi_1+\vartheta(\phi_1-\phi_2)\big]^{5-\alpha}(\phi_1-\phi_2)}{|x-y|^\alpha}dy\Big)(V+\phi_2)^{5-\alpha}\\
  &+(5-\alpha)\Big(\int_{\Omega_\varepsilon}\frac{(V+\phi_1)^{6-\alpha}(y)}{|x-y|^\alpha}dy\Big)\big[V+\phi_1+\vartheta(\phi_1-\phi_2)\big]^{4-\alpha}(\phi_1-\phi_2)\\
  &-(5-\alpha)\displaystyle\Big(\int_{\Omega_\varepsilon}\frac{V^{6-\alpha}(y)}{|x-y|^\alpha}dy\Big)V^{4-\alpha}(\phi_1-\phi_2)\\
  &+(6-\alpha)\Big(\int_{\Omega_\varepsilon}\frac{\big[V+\phi_1+\vartheta(\phi_1-\phi_2)\big]^{5-\alpha}(\phi_1-\phi_2)}{|x-y|^\alpha}dy\Big)(V+\phi_2)^{5-\alpha}\\
  &-(6-\alpha)\displaystyle\Big(\int_{\Omega_\varepsilon}\frac{V^{5-\alpha}(y)(\phi_1-\phi_2)}{|x-y|^\alpha}dy\Big)V^{5-\alpha}.
\end{align*}
For any $\varphi\in H_0^1(\Omega_\varepsilon)$, by the mean value theorem, using \eqref{HLS}, the H\"{o}lder and Sobolev inequalities, we have
\begin{align*}
  &\Big|\int_{\Omega_\varepsilon}\big[N(\phi_1)-N(\phi_2)\big]\varphi dx\Big|\\
  \leq &
  \bigg|\int_{\Omega_\varepsilon}\int_{\Omega_\varepsilon}\frac{(V+\phi_1)^{6-\alpha}(y)\big[(V+\phi_1)^{5-\alpha}-(V+\phi_2)^{5-\alpha}\big](x)\varphi(x)}{|x-y|^\alpha}dxdy\\
  &-(5-\alpha)\int_{\Omega_\varepsilon}\int_{\Omega_\varepsilon}\frac{(V+\phi_1)^{6-\alpha}(y)\big[V+\phi_1+\vartheta(\phi_1-\phi_2)\big]^{4-\alpha}(x)(\phi_1-\phi_2)(x)\varphi(x)}{|x-y|^\alpha}dxdy
  \\&+\int_{\Omega_\varepsilon}\int_{\Omega_\varepsilon}\frac{\big[(V+\phi_1)^{6-\alpha}-(V+\phi_2)^{6-\alpha}\big](y)(V+\phi_2)^{5-\alpha}(x)\varphi(x)}{|x-y|^\alpha}dxdy\\
  &-(6-\alpha)\int_{\Omega_\varepsilon}\int_{\Omega_\varepsilon}\frac{\big[V+\phi_1+\vartheta(\phi_1-\phi_2)\big]^{5-\alpha}(y)(\phi_1-\phi_2)(y)(V+\phi_2)^{5-\alpha}(x)\varphi(x)}{|x-y|^\alpha}dxdy\\
  &+(5-\alpha)\int_{\Omega_\varepsilon}\int_{\Omega_\varepsilon}\frac{(V+\phi_1)^{6-\alpha}(y)\big[V+\phi_1+\vartheta(\phi_1-\phi_2)\big]^{4-\alpha}(x)(\phi_1-\phi_2)(x)\varphi(x)}{|x-y|^\alpha}dxdy\\
  &-(5-\alpha)\int_{\Omega_\varepsilon}\int_{\Omega_\varepsilon}\frac{V^{6-\alpha}(y)V^{4-\alpha}(x)(\phi_1-\phi_2)(x)\varphi(x)}{|x-y|^\alpha}dxdy\\
  &+(6-\alpha)\int_{\Omega_\varepsilon}\int_{\Omega_\varepsilon}\frac{\big[V+\phi_1+\vartheta(\phi_1-\phi_2)\big]^{5-\alpha}(y)(\phi_1-\phi_2)(y)(V+\phi_2)^{5-\alpha}(x)\varphi(x)}{|x-y|^\alpha}dxdy\\
  &-(6-\alpha)\int_{\Omega_\varepsilon}\int_{\Omega_\varepsilon}\frac{V^{5-\alpha}(y)(\phi_1-\phi_2)(y)V^{5-\alpha}(x)\varphi(x)}{|x-y|^\alpha}dxdy\bigg|\\
  \leq &C\bigg|\int_{\Omega_\varepsilon}\int_{\Omega_\varepsilon}\frac{(V+\phi_1)^{6-\alpha}(y)\big[V+\phi_1+\kappa(\phi_1-\phi_2)\big]^{3-\alpha}(x)(\phi_1-\phi_2)^2(x)\varphi(x)}{|x-y|^\alpha}dxdy\bigg|\\
  &+C\bigg|\int_{\Omega_\varepsilon}\int_{\Omega_\varepsilon}\frac{\big[V+\phi_1+\kappa(\phi_1-\phi_2)\big]^{4-\alpha}(y)(\phi_1-\phi_2)^2(y)(V+\phi_2)^{5-\alpha}(x)\varphi(x)}{|x-y|^\alpha}dxdy\bigg|\\
  &+C\bigg|\int_{\Omega_\varepsilon}\int_{\Omega_\varepsilon}\frac{V^{6-\alpha}(y)\big[V+\phi_1+\kappa(\phi_1-\phi_2)\big]^{3-\alpha}(x)(\phi_1-\phi_2)^2(x)\varphi(x)}{|x-y|^\alpha}dxdy\bigg|\\
  &+C\bigg|\int_{\Omega_\varepsilon}\int_{\Omega_\varepsilon}\frac{\big[V+\phi_1+\kappa(\phi_1-\phi_2)\big]^{4-\alpha}(y)(\phi_1-\phi_2)^2(y)V^{5-\alpha}(x)\varphi(x)}{|x-y|^\alpha}dxdy\bigg|\\
  \leq &C\|\phi_1-\phi_2\|^2_{H_0^1(\Omega_\varepsilon)}\|\varphi\|_{H_0^1(\Omega_\varepsilon)},
\end{align*}
where $\vartheta,\kappa \in (0,1)$.
Therefore, for any $\phi_1,\phi_2\in \mathcal{B}$, we have
\begin{align*}
  \|T(\phi_1)-T(\phi_2)\|_{H_0^1(\Omega_\varepsilon)}\leq C\|\phi_1-\phi_2\|^2_{H_0^1(\Omega_\varepsilon)}\leq&  C(\|\phi_1\|_{H_0^1(\Omega_\varepsilon)}+\|\phi_1\|_{H_0^1(\Omega_\varepsilon)})\|\phi_1-\phi_2\|_{H_0^1(\Omega_\varepsilon)}\\
  <&\frac{1}{2}\|\phi_1-\phi_2\|_{H_0^1(\Omega_\varepsilon)}.
\end{align*}

Therefore,  by the contraction mapping theorem, we conclude the result. Finally, using the implicit function theorem, we can prove the regularity of  $\phi_{\mu',\xi'}$. Thus we complete the proof.
\qed

\section{$C^1$-estimate}\label{C^1}
It is important, for later purposes, to understand the differentiability of $\phi_{\mu',\xi'}$ (which is given in Proposition \ref{fixed}) with respect to the variables $\mu'$ and $\xi'_i$, $i=1,2,3$, for a fixed $\varepsilon>0$. We have the following result.
\begin{lemma}\label{ti}
Under the conditions of Lemma \ref{non}, the derivative $\nabla_{\mu',\xi'}\partial _{\mu'}\phi_{\mu',\xi'}$ exists and is a continuous function. Besides, we have
\begin{equation*}
  \|\nabla_{\mu',\xi'}\phi_{\mu',\xi'}\|_{H_0^1(\Omega_\varepsilon)}+\|\nabla_{\mu',\xi'}\partial _{\mu'}\phi_{\mu',\xi'}\|_{H_0^1(\Omega_\varepsilon)}\leq C\varepsilon.
\end{equation*}
\end{lemma}
\begin{proof}
Let us consider differentiation with respect to $\xi'_i$, $i=1,2,3$. For notational simplicity, we write $X_i:=\partial_{\xi'_i}$. Then from \eqref{zhuan}, we have
\begin{align}\label{tidu}
  \Pi^\perp_{\mu',\xi'}L(X_i)=&\Pi^\perp_{\mu',\xi'}\bigg\{%(6-\alpha)(5-\alpha)\Big(\int_{\Omega_\varepsilon}\frac{V^{5-\alpha}(y)\phi(y)}{|x-y|^\alpha}dy\Big)V^{4-\alpha} \partial_{\xi'_i}V
%(5-\alpha)(6-\alpha)\Big(\int_{\Omega_\varepsilon}\frac{V^{5-\alpha}(y)(\partial_{\xi'_i}V)(y)}{|x-y|^\alpha}dy\Big)V^{4-\alpha}\phi\\
% &+(5-\alpha)(4-\alpha)\Big(\int_{\Omega_\varepsilon}\frac{V^{6-\alpha}(y)}{|x-y|^\alpha}dy\Big)V^{3-\alpha}(\partial_{\xi'_i}V)\phi
  (6-\alpha)\Big(\int_{\Omega_\varepsilon}\frac{(V+\phi_{\mu',\xi'})^{5-\alpha}(y)(\partial_{\xi'_i}V+X_i)(y)}{|x-y|^\alpha}dy\Big)(V+\phi_{\mu',\xi'})^{5-\alpha} \nonumber\\
  &+(5-\alpha)\Big(\int_{\Omega_\varepsilon}\frac{(V+\phi_{\mu',\xi'})^{6-\alpha}(y)}{|x-y|^\alpha}dy\Big)(V+\phi_{\mu',\xi'})^{4-\alpha}(\partial_{\xi'_i}V+X_i)\nonumber\\
  %&-(6-\alpha)\Big(\int_{\Omega_\varepsilon}\frac{V^{5-\alpha}(y)(\partial_{\xi'_i}V)(y)}{|x-y|^\alpha}dy\Big)V^{5-\alpha}\\&-(5-\alpha)\Big(\int_{\Omega_\varepsilon}\frac{V^{6-\alpha}(y)}{|x-y|^\alpha}dy\Big)V^{4-\alpha}\partial_{\xi'_i}V\\
 % &-(6-\alpha)(5-\alpha)\Big(\int_{\Omega_\varepsilon}\frac{V^{4-\alpha}(y)(\partial_{\xi'_i}V)(y)\phi(y)}{|x-y|^\alpha}dy\Big)V^{5-\alpha}\\
  &-(6-\alpha)\Big(\int_{\Omega_\varepsilon}\frac{V^{5-\alpha}(y)X_i(y)}{|x-y|^\alpha}dy\Big)V^{5-\alpha}%\\&-(6-\alpha)(5-\alpha)\Big(\int_{\Omega_\varepsilon}\frac{V^{5-\alpha}(y)\phi(y)}{|x-y|^\alpha}dy\Big)V^{4-\alpha}\partial_{\xi'_i}V
  %&-(5-\alpha)(6-\alpha)\Big(\int_{\Omega_\varepsilon}\frac{V^{5-\alpha}(y)(\partial_{\xi'_i}V)(y)}{|x-y|^\alpha}dy\Big)V^{4-\alpha}\phi\\
  %&-(5-\alpha)(4-\alpha)\Big(\int_{\Omega_\varepsilon}\frac{V^{6-\alpha}(y)}{|x-y|^\alpha}dy\Big)V^{3-\alpha}\phi\partial_{\xi'_i}V\\
  -(5-\alpha)\displaystyle\Big(\int_{\Omega_\varepsilon}\frac{V^{6-\alpha}(y)}{|x-y|^\alpha}dy\Big)V^{4-\alpha}X_i \nonumber\\
  %&+(6-\alpha)\Big(\int_{\Omega_\varepsilon}\frac{V^{5-\alpha}(y)(\partial_{\xi'_i}V)(y)}{|x-y|^\alpha}dy\Big)V^{5-\alpha}\\
  %&+(5-\alpha)\Big(\int_{\Omega_\varepsilon}\frac{V^{6-\alpha}(y)}{|x-y|^\alpha}dy\Big)V^{4-\alpha}\partial_{\xi'_i}V\\
  &-(6-\alpha)\Big(\int_{\Omega_\varepsilon}\frac{w_{\mu',\xi'}^{5-\alpha}(y)(\partial_{\xi'_i}w_{\mu',\xi'})(y)}{|x-y|^\alpha}dy\Big)w_{\mu',\xi'}^{5-\alpha}-(5-\alpha)\Big(\int_{\Omega_\varepsilon}\frac{w_{\mu',\xi'}^{6-\alpha}(y)}{|x-y|^\alpha}dy\Big)w_{\mu',\xi'}^{4-\alpha}\partial_{\xi'_i}w_{\mu',\xi'}
  \bigg\}\nonumber\\
  \leq &C\Pi^\perp_{\mu',\xi'}\bigg\{(6-\alpha)\Big(\int_{\Omega_\varepsilon}\frac{(V+\phi_{\mu',\xi'})^{5-\alpha}(y)(\partial_{\xi'_i}V)(y)}{|x-y|^\alpha}dy\Big)(V+\phi_{\mu',\xi'})^{5-\alpha}\nonumber\\
  &+(5-\alpha)\Big(\int_{\Omega_\varepsilon}\frac{(V+\phi_{\mu',\xi'})^{6-\alpha}(y)}{|x-y|^\alpha}dy\Big)(V+\phi_{\mu',\xi'})^{4-\alpha}\partial_{\xi'_i}V\nonumber\\
  &+\Big(\int_{\Omega_\varepsilon}\frac{(V+\phi_{\mu',\xi'})^{5-\alpha}(y)X_i(y)}{|x-y|^\alpha}dy\Big)V^{4-\alpha}\phi_{\mu',\xi'}
  +\Big(\int_{\Omega_\varepsilon}\frac{V^{4-\alpha}(y)\phi_{\mu',\xi'}(y)X_i(y)}{|x-y|^\alpha}dy\Big)V^{5-\alpha}
 \nonumber \\&+\Big(\int_{\Omega_\varepsilon}\frac{(V+\phi_{\mu',\xi'})^{6-\alpha}(y)}{|x-y|^\alpha}dy\Big)V^{3-\alpha}\phi_{\mu',\xi'} X_i+\Big(\int_{\Omega_\varepsilon}\frac{V^{5-\alpha}(y)\phi_{\mu',\xi'}(y)}{|x-y|^\alpha}dy\Big)V^{4-\alpha}X_i\nonumber\\
  &-(6-\alpha)\Big(\int_{\Omega_\varepsilon}\frac{w_{\mu',\xi'}^{5-\alpha}(y)(\partial_{\xi'_i}w_{\mu',\xi'})(y)}{|x-y|^\alpha}dy\Big)w_{\mu',\xi'}^{5-\alpha}-(5-\alpha)\Big(\int_{\Omega_\varepsilon}\frac{w_{\mu',\xi'}^{6-\alpha}(y)}{|x-y|^\alpha}dy\Big)w_{\mu',\xi'}^{4-\alpha}\partial_{\xi'_i}w_{\mu',\xi'}
  \bigg\}
  .
\end{align}
For any $\varphi\in H_0^1(\Omega_\varepsilon)$, using \eqref{HLS}, the H\"{o}lder and Sobolev inequalities, we have
\begin{equation*}
 \bigg|\int_{\Omega_\varepsilon}\int_{\Omega_\varepsilon}\frac{(V+\phi_{\mu',\xi'})^{5-\alpha}(y)X_i(y)V^{4-\alpha}(x)\phi_{\mu',\xi'}(x)\varphi(x)}{|x-y|^\alpha}dxdy\bigg|\leq C\|\phi_{\mu',\xi'}\|_{H_0^1(\Omega_\varepsilon)}\|X_i\|_{H_0^1(\Omega_\varepsilon)}\|\varphi\|_{H_0^1(\Omega_\varepsilon)},
\end{equation*}
\begin{equation*}
 \bigg|\int_{\Omega_\varepsilon}\int_{\Omega_\varepsilon}\frac{V^{4-\alpha}(y)\phi_{\mu',\xi'}(y)X_i(y)V^{5-\alpha}(x)\varphi(x)}{|x-y|^\alpha}dxdy\bigg|\leq C\|\phi_{\mu',\xi'}\|_{H_0^1(\Omega_\varepsilon)}\|X_i\|_{H_0^1(\Omega_\varepsilon)}\|\varphi\|_{H_0^1(\Omega_\varepsilon)},
\end{equation*}
and
\begin{equation*}
 \bigg| \int_{\Omega_\varepsilon}\int_{\Omega_\varepsilon}\frac{(V+\phi_{\mu',\xi'})^{6-\alpha}(y)V^{3-\alpha}(x)\phi_{\mu',\xi'}(x)X_i(x)\varphi(x)}{|x-y|^\alpha}dxdy\bigg|\leq C\|\phi_{\mu',\xi'}\|_{H_0^1(\Omega_\varepsilon)}\|X_i\|_{H_0^1(\Omega_\varepsilon)}\|\varphi\|_{H_0^1(\Omega_\varepsilon)},
\end{equation*}
\begin{equation*}
\bigg|  \int_{\Omega_\varepsilon}\int_{\Omega_\varepsilon}\frac{V^{5-\alpha}(y)\phi_{\mu',\xi'}(y)V^{4-\alpha}(x)X_i(x)\varphi(x)}{|x-y|^\alpha}dxdy\bigg|\leq C\|\phi_{\mu',\xi'}\|_{H_0^1(\Omega_\varepsilon)}\|X_i\|_{H_0^1(\Omega_\varepsilon)}\|\varphi\|_{H_0^1(\Omega_\varepsilon)}.
\end{equation*}
This with $\|\phi_{\mu',\xi'}\|_{H_0^1(\Omega_\varepsilon)}\leq C\varepsilon$ yields that
\begin{align}\label{tidu1}
  \bigg\|&\Big(\int_{\Omega_\varepsilon}\frac{(V+\phi_{\mu',\xi'})^{5-\alpha}(y)X_i(y)}{|x-y|^\alpha}dy\Big)V^{4-\alpha}\phi_{\mu',\xi'}
  +\Big(\int_{\Omega_\varepsilon}\frac{V^{4-\alpha}(y)\phi_{\mu',\xi'}(y)X_i(y)}{|x-y|^\alpha}dy\Big)V^{5-\alpha}
  \nonumber\\&+\Big(\int_{\Omega_\varepsilon}\frac{(V+\phi_{\mu',\xi'})^{6-\alpha}(y)}{|x-y|^\alpha}dy\Big)V^{3-\alpha}\phi _{\mu',\xi'} X_i+\Big(\int_{\Omega_\varepsilon}\frac{V^{5-\alpha}(y)\phi_{\mu',\xi'}(y)}{|x-y|^\alpha}dy\Big)V^{4-\alpha}X_i\bigg\|_{H_0^1(\Omega_\varepsilon)}
  \nonumber\\ \leq& C\|\phi_{\mu',\xi'}\|_{H_0^1(\Omega_\varepsilon)}\|X_i\|_{H_0^1(\Omega_\varepsilon)}\leq C\varepsilon \|X_i\|_{H_0^1(\Omega_\varepsilon)}.
\end{align}
Moreover, for any $\varphi\in H_0^1(\Omega_\varepsilon)$, using \eqref{HLS}, \eqref{gu3}, the H\"{o}lder and Sobolev inequalities, we have
\begin{align*}
  &\bigg|\int_{\Omega_\varepsilon}\int_{\Omega_\varepsilon}\frac{(V+\phi_{\mu',\xi'})^{5-\alpha}(y)(\partial_{\xi'_i}V)(y)(V+\phi_{\mu',\xi'})^{5-\alpha}(x)\varphi(x)}{|x-y|^\alpha}dxdy\\
  &-\int_{\Omega_\varepsilon}\int_{\Omega_\varepsilon}
  \frac{w_{\mu',\xi'}^{5-\alpha}(y)(\partial_{\xi'_i}w_{\mu',\xi'})(y)w_{\mu',\xi'}^{5-\alpha}(x)\varphi(x)}{|x-y|^\alpha}dxdy\bigg|\\
  \leq & C\bigg|\int_{\Omega_\varepsilon}\int_{\Omega_\varepsilon}
  \frac{w_{\mu',\xi'}^{5-\alpha}(y)\frac{\partial [\varepsilon^{1/2}\pi_{\mu,\xi}(\varepsilon y)]}{\partial \xi'_i}w_{\mu',\xi'}^{5-\alpha}(x)\varphi(x)}{|x-y|^\alpha}dxdy\bigg|
  \leq  C\varepsilon^2\|\varphi\|_{H_0^1(\Omega_\varepsilon)},
\end{align*}
and
\begin{align*}
  &\bigg|\int_{\Omega_\varepsilon}\int_{\Omega_\varepsilon}\frac{(V+\phi_{\mu',\xi'})^{6-\alpha}(y)(V+\phi_{\mu',\xi'})^{4-\alpha}(x)(\partial_{\xi'_i}V)(x)\varphi(x)}{|x-y|^\alpha}dxdy\\
  &-\int_{\Omega_\varepsilon}\int_{\Omega_\varepsilon}
  \frac{w_{\mu',\xi'}^{6-\alpha}(y)w_{\mu',\xi'}^{4-\alpha}(x)(\partial_{\xi'_i}w_{\mu',\xi'})(x)\varphi(x)}{|x-y|^\alpha}dxdy\bigg|\\
  \leq& C\bigg|\int_{\Omega_\varepsilon}\int_{\Omega_\varepsilon}
  \frac{w_{\mu',\xi'}^{6-\alpha}(y)w_{\mu',\xi'}^{4-\alpha}(x)\frac{\partial [\varepsilon^{1/2}\pi_{\mu,\xi}(\varepsilon x)]}{\partial \mu'_i}\varphi(x)}{|x-y|^\alpha}dxdy\bigg|\leq  C\varepsilon^2\|\varphi\|_{H_0^1(\Omega_\varepsilon)}.
\end{align*}
Hence, we obtain
\begin{align}\label{tidu2}
  \bigg\|&(6-\alpha)\Big(\int_{\Omega_\varepsilon}\frac{(V+\phi_{\mu',\xi'})^{5-\alpha}(y)(\partial_{\xi'_i}V)(y)}{|x-y|^\alpha}dy\Big)(V+\phi_{\mu',\xi'})^{5-\alpha}\nonumber\\
  &+(5-\alpha)\Big(\int_{\Omega_\varepsilon}\frac{(V+\phi_{\mu',\xi'})^{6-\alpha}(y)}{|x-y|^\alpha}dy\Big)(V+\phi_{\mu',\xi'})^{4-\alpha}\partial_{\xi'_i}V\nonumber\\
  &-(6-\alpha)\Big(\int_{\Omega_\varepsilon}\frac{w_{\mu',\xi'}^{5-\alpha}(y)(\partial_{\xi'_i}w_{\mu',\xi'})(y)}{|x-y|^\alpha}dy\Big)w_{\mu',\xi'}^{5-\alpha}
 \nonumber\\& -(5-\alpha)\Big(\int_{\Omega_\varepsilon}\frac{w_{\mu',\xi'}^{6-\alpha}(y)}{|x-y|^\alpha}dy\Big)w_{\mu',\xi'}^{4-\alpha}\partial_{\xi'_i}w_{\mu',\xi'}\bigg\|_{H_0^1(\Omega_\varepsilon)}\leq C\varepsilon^2.
\end{align}
The conclusion follows from \eqref{tidu}-\eqref{tidu2} and Lemma \ref{xian}. The corresponding result for differentiation with respect to $\mu'$ follows similarly. This finishes the proof.
\end{proof}

We shall next analyse the differentiability of  $N(\phi_{\mu',\xi'})$ with respect to the variables $\mu'$ and $\xi'_i$, $i=1,2,3$.
\begin{lemma}\label{non'}
Under the conditions of Lemma \ref{non},
there holds
\begin{equation*}
  \|\nabla_{\mu',\xi'}N(\phi_{\mu',\xi'})\|_{H_0^1(\Omega_\varepsilon)}+\|\nabla_{\mu',\xi'}\partial _{\mu'}N(\phi_{\mu',\xi'})\|_{H_0^1(\Omega_\varepsilon)}\leq C\varepsilon.
\end{equation*}
\end{lemma}
\begin{proof}
Let us consider differentiation with respect to $\xi'_i$, $i=1,2,3$. Then by the definition of $N(\phi_{\mu',\xi'})$, we have
\begin{align*}
  &\partial_{\xi'_i}N(\phi_{\mu',\xi'})\\=&(6-\alpha)\Big(\int_{\Omega_\varepsilon}\frac{(V+\phi_{\mu',\xi'})^{5-\alpha}(y)X_i(y)}{|x-y|^\alpha}dy\Big)(V+\phi_{\mu',\xi'})^{5-\alpha}-(6-\alpha)\Big(\int_{\Omega_\varepsilon}\frac{V^{5-\alpha}(y)X_i(y)}{|x-y|^\alpha}dy\Big)V^{5-\alpha}\\&
  +(5-\alpha)\Big(\int_{\Omega_\varepsilon}\frac{(V+\phi_{\mu',\xi'})^{6-\alpha}(y)}{|x-y|^\alpha}dy\Big)(V+\phi_{\mu',\xi'})^{4-\alpha}X_i
  %&-(6-\alpha)\Big(\int_{\Omega_\varepsilon}\frac{V^{5-\alpha}(y)(\partial_{\xi'_i}V)(y)}{|x-y|^\alpha}dy\Big)V^{5-\alpha}-
  -(5-\alpha)\Big(\int_{\Omega_\varepsilon}\frac{V^{6-\alpha}(y)}{|x-y|^\alpha}dy\Big)V^{4-\alpha}X_i\\
  &+(6-\alpha)\Big(\int_{\Omega_\varepsilon}\frac{(V+\phi_{\mu',\xi'})^{5-\alpha}(y)(\partial_{\xi'_i} V)(y)}{|x-y|^\alpha}dy\Big)(V+\phi_{\mu',\xi'})^{5-\alpha}\\
  &-(6-\alpha)\Big(\int_{\Omega_\varepsilon}\frac{V^{5-\alpha}(y)(\partial_{\xi'_i} V)(y)}{|x-y|^\alpha}dy\Big)V^{5-\alpha}
-(6-\alpha)(5-\alpha)\Big(\int_{\Omega_\varepsilon}\frac{V^{4-\alpha}(y)(\partial_{\xi'_i}V)(y)\phi_{\mu',\xi'}(y)}{|x-y|^\alpha}dy\Big)V^{5-\alpha}\\&
  -(6-\alpha)(5-\alpha)\Big(\int_{\Omega_\varepsilon}\frac{V^{5-\alpha}(y)(\partial_{\xi'_i}V)(y)}{|x-y|^\alpha}dy\Big)V^{4-\alpha}\phi_{\mu',\xi'}\\&
  +(5-\alpha)\Big(\int_{\Omega_\varepsilon}\frac{(V+\phi_{\mu',\xi'})^{6-\alpha}(y)}{|x-y|^\alpha}dy\Big)(V+\phi_{\mu',\xi'})^{4-\alpha}\partial_{\xi'_i} V\\
  %&-(6-\alpha)\Big(\int_{\Omega_\varepsilon}\frac{V^{5-\alpha}(y)(\partial_{\xi'_i}V)(y)}{|x-y|^\alpha}dy\Big)V^{5-\alpha}-
  &-(5-\alpha)\Big(\int_{\Omega_\varepsilon}\frac{V^{6-\alpha}(y)}{|x-y|^\alpha}dy\Big)V^{4-\alpha}\partial_{\xi'_i} V
  %&-(6-\alpha)\Big(\int_{\Omega_\varepsilon}\frac{V^{5-\alpha}(y)X_i(y)}{|x-y|^\alpha}dy\Big)V^{5-\alpha}-
  -(6-\alpha)(5-\alpha)\Big(\int_{\Omega_\varepsilon}\frac{V^{5-\alpha}(y)\phi_{\mu',\xi'}(y)}{|x-y|^\alpha}dy\Big)V^{4-\alpha}\partial_{\xi'_i}V\\
  &-(5-\alpha)(4-\alpha)\Big(\int_{\Omega_\varepsilon}\frac{V^{6-\alpha}(y)}{|x-y|^\alpha}dy\Big)V^{3-\alpha}\phi_{\mu',\xi'}\partial_{\xi'_i}V.%-(5-\alpha)\Big(\int_{\Omega_\varepsilon}\frac{V^{6-\alpha}(y)}{|x-y|^\alpha}dy\Big)V^{4-\alpha}X_i.
\end{align*}
Hence, for any $\varphi\in H_0^1(\Omega_\varepsilon)$, similar to Lemma \ref{non}, we have
\begin{align*}
  \Big|\int_{\Omega_\varepsilon}\partial_{\xi'_i}N(\phi_{\mu',\xi'})\varphi dx\Big|\leq C\|\phi_{\mu',\xi'}\|_{H_0^1(\Omega_\varepsilon)}\|X_i\|_{H_0^1(\Omega_\varepsilon)}\|\varphi\|_{H_0^1(\Omega_\varepsilon)}+C\|\phi_{\mu',\xi'}\|^2_{H_0^1(\Omega_\varepsilon)}\|\varphi\|_{H_0^1(\Omega_\varepsilon)}.
\end{align*}
This with Lemma \ref{ti} yields that $\|\partial_{\xi'_i}N(\phi_{\mu',\xi'})\|_{H_0^1(\Omega_\varepsilon)}\leq C\varepsilon$. The corresponding result for differentiation with respect to $\mu'$ follows similarly.
\end{proof}

\section{Proof of Theorem \ref{th}}\label{Final}
Let us consider the situation in Theorem \ref{th}. Assume the situation $(a)$ of local minimizer
\begin{equation*}
  0=\inf\limits_{\mathfrak{D}}g_{\lambda_0}<\inf\limits_{\partial\mathfrak{D}}g_{\lambda_0}.
\end{equation*}
Then for $\lambda$ close to $\lambda_0$ and $\lambda>\lambda_0$, we have
\begin{equation*}
  \inf\limits_{ \mathfrak{D}}g_{\lambda}<-A(\lambda-\lambda_0),\quad A>0.
\end{equation*}
Let us consider the shrinking set
\begin{equation*}
  \mathfrak{D}_\lambda=\Big\{x\in \mathfrak{D}:g_\lambda(x)<-\frac{A}{2}(\lambda-\lambda_0)\Big\}.
\end{equation*}
Assume
$\lambda>\lambda_0$ is sufficiently close to $\lambda_0$, then $g_\lambda=-\frac{A}{2}(\lambda-\lambda_0)$ on $\partial \mathfrak{D}_\lambda$.

Now, let us consider the situation of part $(b)$. Since $g_\lambda(\xi)$ has a non-degenerate critical point at $\lambda=\lambda_0$ and $\xi=\xi_0$, this is also the case at a certain critical point $\xi_\lambda$ for all $\lambda$ close to $\lambda_0$, where $|\xi_\lambda-\xi_0|=O(\lambda-\lambda_0)$. Moreover, for some intermediate point $\tilde{\xi}_\lambda$, there holds
\begin{equation*}
  g_\lambda(\xi_\lambda)=g_\lambda(\xi_0)+Dg_\lambda(\tilde{\xi}_\lambda)(\xi_\lambda-\xi_0)\geq A(\lambda-\lambda_0)+o(\lambda-\lambda_0),
\end{equation*}
for a certain $A>0$. Let us consider the ball $B_\rho^\lambda$ with center $\xi_\lambda$ and radius $\rho(\lambda-\lambda_0)$ for fixed and small $\rho>0$. Then we have that $g_\lambda(\xi)>\frac{A}{2}(\lambda-\lambda_0)$ for all $\xi\in B_\rho^\lambda$. In this situation, we set $ \mathfrak{D}_\lambda=B_\rho^\lambda$.

In is convenient to introduce the following relabeling of the parameter $\mu$. Let us set
\begin{equation}\label{mu}
  \mu=-\frac{a_1}{2a_2}\frac{g_\lambda(\xi)}{\lambda}\Lambda,
\end{equation}
where $\xi\in \mathfrak{D}_\lambda$ and $a_1,a_2$ are the constants given by \eqref{changshu}. We have the following result, which was proved in \cite[Lemma 3.3]{dDM}.

\begin{lemma}\label{point}
Assume the validity of one of the conditions $(a)$ or $(b)$ of Theorem \ref{th}, and consider a functional of the form:
\begin{equation*}
  \Psi_\lambda(\Lambda,\xi)=\mathcal{J}_\lambda(U_{\mu,\xi})+g^2_\lambda(\xi)\theta_\lambda(\Lambda,\xi),
\end{equation*}
where $\mu$ is given by \eqref{mu}. Denote $\nabla=(\partial _\Lambda,\partial_\xi)$, for any given $\delta>0$, assume that
\begin{equation*}
  |\theta_\lambda|+|\nabla \theta_\lambda|+|\nabla \partial _\Lambda \theta_\lambda|\rightarrow0, \quad \text{as $\lambda\rightarrow\lambda_0$},
\end{equation*}
uniformly on $\xi\in \mathfrak{D}_\lambda$ and $\Lambda\in (\delta,\delta^{-1})$. Then $\Psi_\lambda$ has a critical point $(\Lambda_\lambda,\xi_\lambda)$ with $\xi_\lambda\in \mathfrak{D}_\lambda$, $\Lambda_\lambda\rightarrow 1$.
\end{lemma}

For $\phi_{\mu',\xi'}$ given in Proposition \ref{fixed}, we define
\begin{equation*}
  \tilde{\mathcal{I}}_\lambda(\mu',\xi')=\mathcal{I}_\lambda(V+\phi_{\mu',\xi'}).
\end{equation*}
Then from \cite[Lemma 3.2]{YZ}, we have the following lemma.
\begin{lemma}
Under the conditions of Lemma \ref{non}, point $(\mu',\xi')$ is a critical point of $\tilde{\mathcal{I}}_\lambda(\mu',\xi')$ if and only if $V+\phi_{\mu',\xi'}$ is a critical point of $\mathcal{I}_\lambda(v)$.
\end{lemma}
In the following lemma, we find an expansion for $\tilde{\mathcal{I}}_\lambda(\mu',\xi')$.
\begin{lemma}\label{fiex}
Under the conditions of Lemma \ref{non}, the following expansion holds:%
\begin{equation*}
  \tilde{\mathcal{I}}_\lambda(\mu',\xi')=\mathcal{I}_\lambda(V)+\varepsilon^2 \theta(\mu',\xi'),
\end{equation*}
where %
\begin{equation*}
  |\theta|+|\nabla_{\mu',\xi'}\theta|+|\nabla_{\mu',\xi'}\partial _{\mu'}\theta|\leq C.
\end{equation*}
\end{lemma}
\begin{proof}
By Proposition \ref{fixed}, we know $D \mathcal{I}_\lambda(V+\phi_{\mu',\xi'})[\phi_{\mu',\xi'}]=0$. A Taylor expansions gives
\begin{align*}
 & \mathcal{I}_\lambda(V+\phi_{\mu',\xi'})-\mathcal{I}_\lambda(V)\\
  =&-\int_{0}^1sD^2 \mathcal{I}_\lambda(V+s\phi_{\mu',\xi'})[\phi^2_{\mu',\xi'}]ds\\
  =&-\int_{0}^1s\Big(\int_{\Omega_\varepsilon}\big[N(\phi_{\mu',\xi'})+E\big]\phi _{\mu',\xi'} dx+(6-\alpha)\int_{\Omega_\varepsilon}\int_{\Omega_\varepsilon}\frac{V^{5-\alpha}(y)\phi_{\mu',\xi'}(y)V^{5-\alpha}(x)\phi_{\mu',\xi'}(x)}{|x-y|^\alpha}dxdy\\
  &-(6-\alpha)\int_{\Omega_\varepsilon}\int_{\Omega_\varepsilon}\frac{(V+s\phi_{\mu',\xi'})^{5-\alpha}(y)\phi_{\mu',\xi'}(y)(V+s\phi_{\mu',\xi'})^{5-\alpha}(x)\phi_{\mu',\xi'}(x)}{|x-y|^\alpha}dxdy\\
  &+(5-\alpha)\int_{\Omega_\varepsilon}\int_{\Omega_\varepsilon}\frac{V^{6-\alpha}(y)V^{4-\alpha}(x)\phi^2_{\mu',\xi'}(x)}{|x-y|^\alpha}dxdy\\
  &-(5-\alpha)\int_{\Omega_\varepsilon}\int_{\Omega_\varepsilon}\frac{(V+s\phi_{\mu',\xi'})^{6-\alpha}(y)(V+s\phi_{\mu',\xi'})^{4-\alpha}(x)\phi^2_{\mu',\xi'}(x)}{|x-y|^\alpha}dxdy
  \Big)ds.
\end{align*}
From Lemmas \ref{non}, \ref{err}, and Proposition \ref{fixed}, using \eqref{HLS},  the H\"{o}lder and Sobolev inequalities, we obtain
\begin{equation*}
  \Big|\int_{\Omega_\varepsilon}\big[N(\phi_{\mu',\xi'})+E\big]\phi _{\mu',\xi'} dx\Big|\leq C\big(\|N(\phi_{\mu',\xi'})\|_{H_0^1(\Omega_\varepsilon)}+\|E\|_{H_0^1(\Omega_\varepsilon)}\big)\|\phi_{\mu',\xi'}\|_{H_0^1(\Omega_\varepsilon)}\leq C\varepsilon^2,
\end{equation*}
\begin{align*}
  &\bigg|\int_{\Omega_\varepsilon}\int_{\Omega_\varepsilon}\frac{V^{5-\alpha}(y)\phi_{\mu',\xi'}(y)V^{5-\alpha}(x)\phi_{\mu',\xi'}(x)}{|x-y|^\alpha}dxdy\\
  &-\int_{\Omega_\varepsilon}\int_{\Omega_\varepsilon}\frac{(V+s\phi_{\mu',\xi'})^{5-\alpha}(y)\phi_{\mu',\xi'}(y)(V+s\phi_{\mu',\xi'})^{5-\alpha}(x)\phi_{\mu',\xi'}(x)}{|x-y|^\alpha}dxdy\bigg|\\
  \leq &C \bigg|\int_{\Omega_\varepsilon}\int_{\Omega_\varepsilon}\frac{w_{\mu',\xi'}^{5-\alpha}(y)\phi_{\mu',\xi'}(y)w_{\mu',\xi'}^{5-\alpha}(x)\phi_{\mu',\xi'}(x)}{|x-y|^\alpha}dxdy\bigg|\\
  \leq& C\|\phi_{\mu',\xi'}\|^2_{H_0^1(\Omega_\varepsilon)}\leq C\varepsilon^2,
\end{align*}
and
\begin{align*}
  \bigg|
  &\int_{\Omega_\varepsilon}\int_{\Omega_\varepsilon}\frac{V^{6-\alpha}(y)V^{4-\alpha}(x)\phi^2_{\mu',\xi'}(x)}{|x-y|^\alpha}dxdy\\
  &-\int_{\Omega_\varepsilon}\int_{\Omega_\varepsilon}\frac{(V+s\phi_{\mu',\xi'})^{6-\alpha}(y)(V+s\phi_{\mu',\xi'})^{4-\alpha}(x)\phi^2_{\mu',\xi'}(x)}{|x-y|^\alpha}dxdy
  \bigg|\\
  \leq &C \bigg|\int_{\Omega_\varepsilon}\int_{\Omega_\varepsilon}\frac{w_{\mu',\xi'}^{6-\alpha}(y)w_{\mu',\xi'}^{4-\alpha}(x)\phi^2_{\mu',\xi'}(x)}{|x-y|^\alpha}dxdy\bigg|\\
  \leq& C\|\phi_{\mu',\xi'}\|^2_{H_0^1(\Omega_\varepsilon)}\leq C\varepsilon^2.
\end{align*}
So we have
\begin{equation*}
  \tilde{\mathcal{I}}_\lambda(\mu',\xi')=\mathcal{I}_\lambda(V)+O(\varepsilon^2).
\end{equation*}

Observe that
\begin{align*}
 & \nabla_{\mu',\xi'}\big[\mathcal{I}_\lambda(V+\phi_{\mu',\xi'})-\mathcal{I}_\lambda(V)\big]\\
  %=&-\int_{0}^1sD^2 \mathcal{I}_\lambda(V+s\phi_{\mu',\xi'})[\phi^2_{\mu',\xi'}]ds\\
  =&-\int_{0}^1s\Big[\int_{\Omega_\varepsilon}\big[N(\phi_{\mu',\xi'})+E\big]\nabla_{\mu',\xi'}\phi _{\mu',\xi'} dx+\int_{\Omega_\varepsilon}\phi _{\mu',\xi'}\nabla_{\mu',\xi'}N(\phi_{\mu',\xi'}) dx\\
  &+(6-\alpha)\int_{\Omega_\varepsilon}\int_{\Omega_\varepsilon}\frac{\nabla_{\mu',\xi'}\big[V^{5-\alpha}(y)\phi_{\mu',\xi'}(y)V^{5-\alpha}(x)\phi_{\mu',\xi'}(x)\big]}{|x-y|^\alpha}dxdy\\
  &-(6-\alpha)\int_{\Omega_\varepsilon}\int_{\Omega_\varepsilon}\frac{\nabla_{\mu',\xi'}\big[(V+s\phi_{\mu',\xi'})^{5-\alpha}(y)\phi_{\mu',\xi'}(y)(V+s\phi_{\mu',\xi'})^{5-\alpha}(x)\phi_{\mu',\xi'}(x)\big]}{|x-y|^\alpha}dxdy\\
  &+(5-\alpha)\int_{\Omega_\varepsilon}\int_{\Omega_\varepsilon}\frac{\nabla_{\mu',\xi'}\big[V^{6-\alpha}(y)V^{4-\alpha}(x)\phi^2_{\mu',\xi'}(x)\big]}{|x-y|^\alpha}dxdy\\
  &-(5-\alpha)\int_{\Omega_\varepsilon}\int_{\Omega_\varepsilon}\frac{\nabla_{\mu',\xi'}\big[(V+s\phi_{\mu',\xi'})^{6-\alpha}(y)(V+s\phi_{\mu',\xi'})^{4-\alpha}(x)\phi^2_{\mu',\xi'}(x)\big]}{|x-y|^\alpha}dxdy
  \Big]ds.
\end{align*}
By Lemmas \ref{non}-\ref{err}, \ref{ti}-\ref{non'}, and Proposition \ref{fixed}, using \eqref{HLS}, \eqref{gu2}, \eqref{gu1},  the H\"{o}lder and Sobolev inequalities, we get
\begin{align*}
  &\big|\nabla_{\mu',\xi'}\big[\mathcal{I}_\lambda(V+\phi_{\mu',\xi'})-\mathcal{I}_\lambda(V)\big]\big|\\ \leq & C\big(\|N(\phi_{\mu',\xi'})\|_{H_0^1(\Omega_\varepsilon)}+\|E\|_{H_0^1(\Omega_\varepsilon)}\big)\|\nabla_{\mu',\xi'}\phi_{\mu',\xi'}\|_{H_0^1(\Omega_\varepsilon)}
  +C\|\phi_{\mu',\xi'}\|_{H_0^1(\Omega_\varepsilon)}\|\nabla_{\mu',\xi'}N(\phi_{\mu',\xi'})\|_{H_0^1(\Omega_\varepsilon)}\\
  &+C\|\phi_{\mu',\xi'}\|^2_{H_0^1(\Omega_\varepsilon)}+C\|\phi_{\mu',\xi'}\|_{H_0^1(\Omega_\varepsilon)}\|\nabla_{\mu',\xi'}\phi_{\mu',\xi'}\|_{H_0^1(\Omega_\varepsilon)}\leq C\varepsilon^2.
\end{align*}
A similar computation yields the result.
\end{proof}
\noindent {\bf Proof of Theorem \ref{th}.} Let us choose $\mu$ as in \eqref{mu},
since $\mu'\in (\delta,\delta^{-1})$ for some $\delta>0$, by Lemma \ref{fiex}, we have
\begin{equation*}
  \tilde{\mathcal{I}}_\lambda(\mu',\xi')=\mathcal{I}_\lambda(V)+g_\lambda^2 \theta(\mu',\xi'),
\end{equation*}
with $|\theta|+|\nabla_{\mu',\xi'}\theta|+|\nabla_{\mu',\xi'}\partial _{\mu'}\theta|\leq C$. Define
\begin{equation*}
\Psi_\lambda(\Lambda,\xi)=\tilde{\mathcal{I}}_\lambda(\mu',\xi'),
\end{equation*}
then we have
\begin{equation*}
  \Psi_\lambda(\Lambda,\xi)=\mathcal{I}_\lambda(V)+g_\lambda^2 \theta(\mu',\xi')=\mathcal{J}_\lambda(U_{\mu,\xi})+g_\lambda^2 \theta(\mu',\xi').
\end{equation*}
In view of Lemma \ref{point}, $\Psi_\lambda$ has a critical point. This concludes the proof.
\qed

\section{Proof of Theorem \ref{th1}}\label{Numan}
Arguing as in Section \ref{Pre},
we define $\Theta_{\mu,\xi}$ to be the unique solution of the problem
\begin{align}\label{pi1}
 \left\{
  \begin{array}{ll}
  -\Delta \Theta_{\mu,\xi}=-\lambda\pi_{\mu,\xi}-\lambda w_{\mu,\xi}-\displaystyle\Big(\int_{\mathbb{R}^3\backslash \Omega}\frac{w_{\mu,\xi}^{6-\alpha}(y)}{|x-y|^\alpha}dy\Big)w_{\mu,\xi}^{5-\alpha},
  \ \  &\mbox{in}\ \Omega,\\
 \displaystyle \frac{\partial \Theta_{\mu,\xi}}{\partial \nu}=-\frac{\partial w_{\mu,\xi}}{\partial \nu},
  \ \  &\mbox{on}\ \partial \Omega.
    \end{array}
    \right.
  \end{align}
Fix a small  positive number $\mu$ and a point $\xi\in \Omega$, we consider a first approximation of the solution of the form:
\begin{equation*}
  U_{\mu,\xi}(x)=w_{\mu,\xi}(x)+\Theta_{\mu,\xi}(x).
\end{equation*}
Then $U=U_{\mu,\xi}$ satisfies the equation
\begin{align*}
 \left\{
  \begin{array}{ll}
  -\Delta U=\displaystyle\Big(\int_{\Omega}\frac{w_{\mu,\xi}^{6-\alpha}(y)}{|x-y|^\alpha}dy\Big)w_{\mu,\xi}^{5-\alpha}-\lambda U,
  \ \  &\mbox{in}\ \Omega,\\
  \displaystyle\frac{\partial U}{\partial \nu}=0,
  \ \  &\mbox{on}\ \partial \Omega.
    \end{array}
    \right.
  \end{align*}
Moreover, using estimates contained in Section \ref{Energy} and \cite[Lemmas 3.1 and 3,2]{dMRW}, one can prove the following results.
\begin{lemma}\label{zhan}
For any $\sigma>0$, as $\mu\rightarrow0$, the following expansion holds:
\begin{equation*}
  \mu^{-1/2}\Theta_{\mu,\xi}(x)=-4\pi3^{1/4}H^\lambda(x,\xi)-\mu \mathcal{D}_0(\mu^{-1}({x-\xi}))+\mu^{2-\sigma}\theta(\mu,x,\xi),
\end{equation*}
where for $i=0,1$, $j=0,1,2$, $i+j\leq 2$, the function $\mu^j\frac{\partial ^{i+j}}{\partial \xi^i \partial \mu^j}\theta(\mu,x,\xi)$ is bounded uniformly on $x\in \Omega$, all small $\mu$ and $\xi$ in compact subsets of $\Omega$.
\end{lemma}
\begin{proof}
We argue as in the proof of Lemma \ref{piex}.
\end{proof}
\begin{lemma}
For any $\sigma>0$, as $\mu\rightarrow0$, the following expansion holds:
\begin{equation*}
  \mathcal{J}_\lambda(U_{\mu,\xi})=a_0+a_1\mu g^\lambda(\xi)-a_2\lambda\mu^2 -a_3\mu^2 (g^\lambda)^2(\xi)+\mu^{\frac{5}{2}-\sigma}\theta(\mu,\xi),
\end{equation*}
where for $i=0,1$, $j=0,1,2$, $i+j\leq 2$, the function $\mu^j\frac{\partial ^{i+j}}{\partial \xi^i \partial \mu^j}\theta(\mu,\xi)$ is bounded uniformly on all small $\mu$ and $\xi$ in compact subsets of $\Omega$. The $a_j$'s are explicit constants, given by \eqref{changshu}.
\end{lemma}
\begin{proof}
We argue as in the proof of Lemma \ref{enex}, using Lemma \ref{zhan}.
\end{proof}

We consider the situation $(a)$ of local maximizer in Theorem \ref{th1}
\begin{equation*}
  0=\sup\limits_{\mathcal{U}}g^{\lambda^0}>\sup\limits_{\partial \mathcal{U}}g^{\lambda^0}.
\end{equation*}
Then for $\lambda$ close to $\lambda^0$ and $\lambda>\lambda^0$, we have
\begin{equation*}
  \sup\limits_{ \mathcal{U}}g^{\lambda}>A(\lambda-\lambda^0),\quad A>0.
\end{equation*}
Define the shrinking set
\begin{equation*}
  \mathcal{U}^\lambda=\Big\{x\in \mathcal{U}:g^\lambda(x)>\frac{A}{2}(\lambda-\lambda^0)\Big\}.
\end{equation*}
Assume
$\lambda>\lambda^0$ is sufficiently close to $\lambda^0$, then $g^\lambda=\frac{A}{2}(\lambda-\lambda^0)$ on $\partial \mathcal{U}^\lambda$.

\medskip

\noindent {\bf Proof of Theorem \ref{th1}.} The proof is similar to that of Theorem \ref{th}, so we omit it.
\qed

\medskip
\subsection*{Acknowledgments}
The research has been supported by Natural Science Foundation of Chongqing, China (No. CSTB2024 NSCQ-LZX0038) and National Natural Science Foundation of China (No. 12371121).

\subsection*{Statements and Declarations}
{\bf Conflict of interest} On behalf of all authors, the corresponding author states that there
is no conflict of interest. The manuscript has no associated data.

\end{document}